\documentclass[reqno]{amsart}
\usepackage{amssymb,latexsym}
\usepackage{amsmath}
\usepackage{amsthm}
\usepackage{graphicx}
\usepackage{hyperref}
\usepackage{titletoc}
\numberwithin{equation}{section}
\newtheorem{theorem}{Theorem}[section]
\newtheorem{proposition}[theorem]{Proposition}
\newtheorem{lemma}[theorem]{Lemma}

\theoremstyle{definition}

\def\XXint#1#2#3{{\setbox0=\hbox{$#1{#2#3}{\int}$}
     \vcenter{\hbox{$#2#3$}}\kern-.5\wd0}}

\begin{document}
\title[fractional Gierer-Meinhardt system]{Multi-bump ground states of the fractional Gierer-Meinhardt system in $\mathbb{R}$}

\author{ Juncheng Wei}
\address{ Juncheng ~Wei,~Department of Mathematics, University of British Columbia,
Vancouver, BC V6T 1Z2, Canada}
\email{jcwei@math.ubc.ca}
\author{ Wen Yang}
\address{ Wen ~Yang,~Department of Mathematics, University of British Columbia,
Vancouver, BC V6T 1Z2, Canada}
\email{wyang@math.ubc.ca}
\date{}\maketitle
\begin{abstract}
In this paper we study ground-states of the fractional Gierer-Meinhardt system on the line, namely the solutions of the problem
\begin{equation*}
\left\{\begin{array}{ll}
(-\Delta)^su+u-\frac{u^2}{v}=0,\quad &\mathrm{in}~\mathbb{R},\\
(-\Delta)^sv+\varepsilon^{2s}v-u^2=0,\quad &\mathrm{in}~\mathbb{R},\\
u,v>0,\quad u,v\rightarrow0~&\mathrm{as}~|x|\rightarrow+\infty.
\end{array}\right.
\end{equation*}
We prove that given any positive integer  $k,$ there exists a solution to this problem for $s\in[\frac12,1)$ exhibiting exactly $k$ bumps in its $u-$component, separated from each other at a distance $O(\varepsilon^{\frac{1-2s}{4s}})$ for $s\in(\frac12,1)$ and $O(|\log\varepsilon|^{\frac12})$ for $s=\frac12$ respectively, whenever $\varepsilon$ is sufficiently small. These bumps resemble the shape of the unique solution of
\begin{equation*}
(-\Delta)^sU+U-U^2=0,\quad 0<U(y)\rightarrow0~\mathrm{as}~|y|\rightarrow\infty.
\end{equation*}
\end{abstract}

\section{Introduction}
In this paper we consider the following fractional Gierer-Meinhardt system in $\mathbb{R}$
\begin{equation}
\label{1.1}
\left\{\begin{array}{lll}
(-\Delta)^su+u-\frac{u^2}{v}=0~&\mathrm{in}~\mathbb{R},\\
(-\Delta)^sv+\varepsilon^{2s}v-u^2=0~&\mathrm{in}~\mathbb{R},\\
\end{array}\right.
\end{equation}
where $(-\Delta)^s,0<s<1,$ denotes the  fractional Laplace operator. (For the definition, See Section 2 below.)

When $s=1$, this is the classical Gierer-Meinhardt system  proposed  by Gierer-Meinhardt  in \cite{gm} in 1972. More precisely they considered the following reaction-diffusion system as a model of biological pattern formation
\begin{align}
\label{1.2}
\left\{\begin{array}{lll}
a_t=d\Delta a-a+\frac{a^2}{h},~&\mathrm{in}~\Omega\times(0,t),\\
h_t=D\Delta h-h+a^2,~&\mathrm{in}~\Omega\times(0,t),\\
\partial_{\nu}a=\partial_{\nu}h=0,~&\mathrm{on}~\partial\Omega\times(0,t),
\end{array}\right.
\end{align}
where $d, D>0$ are diffusion rates, $\Omega \subset \mathbb{R}^n$ is a bounded domain and $\partial_{\nu}$ denotes the derivative in the outer normal direction. The Gierer-Meinhardt system was used in \cite{gm} to model head formation of {\em Hydra}, an animal of a few millimeters in length, made up of approximately 100,000 cells of about fifteen different types. It consists of a "head" region located at one end along its length. Typical experiments with {\em hydra} involve removing part of the "head" region transplanted area is sufficiently far from the (old) head. These observations led to the assumption of the existence of two chemical substances a {\em slowly} diffusing activator and a {\em rapidly} diffusing inhibitor, whose concentrations at the point $x\in\Omega$ and time $t>0$ are represented, respectively, by the quantities $a(x,t)$ and $h(x,t)$. Their diffusion rates, given by the positive constants $d$ and $D$ are then assumed to be that $d\ll D$. The Gierer-Meinhardt system falls within the framework of a theory proposed by Turing \cite{t} in 1952 as a mathematical model for the development of complex organisms from a single cell. He speculated that localized peaks in concentration of chemical substances, known as inducers or morphogens, could be responsible for a group of cells developing differently from the surrounding cells. Turing discovered through linear analysis that a large difference in relative size of diffusivities for activating and inhibiting substances carries instability of the homogeneous, constant steady state, thus leading to the presence of nontrivial, possibly stable stationary configurations. Activator-inhibitor systems have been used widely in the mathematical theory of biological pattern formation \cite{m1, m2}. Substantial research concerning this system has been generated in recent years. We refer the reader to the survey papers \cite{n}, \cite{wsurvey} and the book \cite{wwbook} for the overview of the subject.

\medskip

In the last twenty years there have been intensive research on the existence and stability  of steady states of the Gierer-Meinhardt system (\ref{1.2}) in a bounded domain. It is known that  there are multiple spikes solutions which may be stable. We refer to papers \cite{dkp}, \cite{iww}, \cite{w2}, \cite{wardwei}, \cite{ww2}, \cite{ww3} and the book \cite{wwbook} and the references therein. In the case of the domain being the whole space, after suitable rescaling the steady state problem of (\ref{1.2}) becomes
\begin{equation}
\label{1.7}
\left\{\begin{array}{ll}
\Delta u-u+\frac{u^2}{v}=0,~u>0, \ \mbox{in} \ \mathbb{R}^n\\
\Delta v-\varepsilon^2v+u^2=0,~v>0, \ \mbox{in} \ \mathbb{R}^n\\
u,v\rightarrow0~\mathrm{as}~|x|\rightarrow+\infty.
\end{array}\right.
\end{equation}
A solution to (\ref{1.7}) is called ground state. In the real line case the existence of single and multiple pulse solutions is proved independently by Doelman-Gardner-Kaper  in \cite{dkc} (via geometric dynamical system method) and by Chen-del Pino-Kowalczyk  in \cite{dgk}  (via PDE reduction method). Similar results have been obtained for the case $n=2$ by del Pino-Kowalczyk-Wei in \cite{dkw}. In higher dimensional case another type of solutions exist: solutions which are radially symmetric but have rings concentrations. We refer to Ni-Wei \cite{nw}, Kolokolnikov-Wei \cite{kw} and Kolokolnikov-Wei-Yang \cite{kwy}. In $\mathbb{R}^3$ there exists also axially symmetric solution with smoke ring concentrations. See Kolokolnikov-Ren \cite{kr}. The presence of steady configurations in the whole space appears driven by smallness of the {\em relative size} $\varepsilon^2=\frac{d}{D}$ of the diffusion rates of the activating and inhibiting substances. The present paper deals with the case  of equation (\ref{1.7}) in the nonlocal diffusion--fractional laplacian case (\ref{1.1}) for $s\in[\frac12,1)$. Next we briefly discuss the fractional laplacian and nonlocal diffusion.

In probability, we consider the random walk for L$\acute{e}$vy processes:
$$u_j^{n+1}=\sum_kP_{jk}u_k^n,$$
where $P_{jk}$ denotes the transition function which has a tail (i.e., power decay with the distance $|j-k|$). By taking the limit we get an operator $(-\Delta)^s$ as the infinitesimal generator of a L$\acute{e}$vy process: if $X_t$ is the isotropic $2s-$stable l$\acute{e}$vy process we have
$$(-\Delta)^su(x)=\lim_{h\rightarrow0^+}\frac{1}{h}\mathbb{E}[u(x)-u(x+X_h)].$$
When $s=1$ it corresponds to the Brownian motion.

Fractional diffusion equations have been used to model anomalously slow or fast scattering of particles in a variety of natural applications. A consideration of the problem of anomalous sub-diffusion with reactions in terms of continuous-time random walks (CTRWs) with sources and sinks leads to a fractional activator-inhibitor model with a fractional order temporal derivative operating on the spatial Laplacian. A similar type of system has also been proposed for diffusion with reactions on a fractal. The problem of anomalous super-diffusion with reactions has also been considered and in this case a fractional reaction-diffusion model has been proposed with the spatial Laplacian replaced by a spatial fractional differential operator. If the reaction time is not short compared with the diffusion time in sub-diffusive systems with reactions (for example, if many encounters between reactants are required before reactions proceed), then an alternate model has been proposed where the fractional order temporal derivative operates on both the spatial Laplacian and the reaction term. An important distinction between the two anomalous reaction-diffusion models becomes apparent when the concentration of species is spatially homogeneous, this latter model does not reduce to the classical macroscopic rate equations except when the diffusion is also non-anomalous. For the more background on fractional reaction-diffusion system, we refer the readers to \cite{GMV, hl, mcf} and references therein. We shall mention that weakly nonlinear analysis has been done to constant equilibriums for  Turing's system by Henry-Langlands-Wearne \cite{hl} and Golovin-Matkowsky-Volpert \cite{GMV}. In \cite{Nec} the author studied a slightly different Gierer-Meinhardt system with fractional diffusion
\begin{align}
\label{1.2m}
\left\{\begin{array}{lll}
a_t=-d (- \Delta)^s a-a+\frac{a^2}{h},~&\mathrm{in}~ (-1, 1)\times(0,t),\\
h_t=D\Delta h-h+a^2,~&\mathrm{in}~ (-1, 1)\times(0,t),\\
a_x (\pm 1, t) = h_x (\pm 1, t)=0
\end{array}\right.
\end{align}
where $ (-\Delta)^s $ denotes the spectral fractional Laplacian.

\medskip

 In this paper we shall consider the existence of nonlinear patterns for the classical Gierer-Meinhardt system in the real line (\ref{1.1}) with fractioal diffusion. A similar notable feature of problem (\ref{1.1}) (as the classical case) is, it will be shown in this work, the presence of a large number of solutions (modulo translations) as the parameter $\varepsilon$ gets smaller. More precisely, given any positive integer  $k$ we find a number $\varepsilon_k$ such that if $0<\varepsilon<\varepsilon_k$ then there exists a solution exhibiting exactly $k$ {\em bumps} in the activator. Besides, after appropriate re-scaling of $u$, these bumps are approaching a universal profile and are separately from each other $\varepsilon^{\frac{1-2s}{4s}}$ and $(-\log\varepsilon)^{\frac12}$ for the case $s\in(\frac12,1)$ and $s=\frac12$ respectively. We remark that this phenomenon is intrinsic to the full system, for only one ground state of the following equation modulo translations exists (\cite{fl})
\begin{equation}
\label{1.8}
(-\Delta)^s U+U-U^2=0~\mathrm{in}~\mathbb{R},\quad 0<U(y)\rightarrow0~\mathrm{as}~|y|\rightarrow\infty.
\end{equation}

\medskip

Before stating our main results on the existence of such solutions, we need the following preparations. In the sequel by $U(x)$ we denote the unique radially symmetric solution of (\ref{1.8}). (For the existence and uniqueness, we refer to Frank-Lenzmann \cite{fl}.)  In order to describe the position of bumps,  we set
\begin{align}
\label{1.9}
\Xi_s(q_1,q_2,\cdots,q_m)=~&
\sum_{i=1}^m\Big(\alpha_sU(2q_i)+\beta_s\varepsilon^{2s-1}|2q_i|^{2s-1}\Big)\nonumber\\
&+\sum_{i\neq j}\Big(\alpha_sU(q_i-q_j)+\beta_s\varepsilon^{2s-1}|q_i-q_j|^{2s-1}\Big)\nonumber\\
&+\sum_{i\neq j}\Big(\alpha_sU(q_i+q_j)+\beta_s\varepsilon^{2s-1}|q_i+q_j|^{2s-1}\Big)
\end{align}
for $s\in(\frac12,1)$ and
\begin{align}
\label{1.10}
\Xi_{\frac12}(q_1,q_2,\cdots,q_m)=~&
\sum_{i=1}^m\Big(\alpha_{\frac12}U(2q_i)
+\beta_{\frac12}\frac{1}{\log\frac{1}{\varepsilon}}\log|2q_i|\Big)\nonumber\\
&+\sum_{i\neq j}\Big(\alpha_{\frac12}U(q_i-q_j)
+\beta_{\frac12}\frac{1}{\log\frac{1}{\varepsilon}}\log|q_i-q_j|\Big)\nonumber\\
&+\sum_{i\neq j}\Big(\alpha_{\frac12}U(q_i+q_j)
+\beta_{\frac12}\frac{1}{\log\frac{1}{\varepsilon}}\log|q_i+q_j|\Big)
\end{align}
for $s=\frac12.$

\medskip

In Section 8, we shall prove the function $\Xi_s,~s\in[\frac12,1)$ admits global minimal point in the interior of the following set
\begin{align}
\label{1.11}
Q_{s,\eta}=\Big\{(q_1,q_2,\cdots,q_m)\mid
\frac{1}{\eta}\varepsilon^{\frac{1-2s}{4s}}>q_i>\eta\varepsilon^{\frac{1-2s}{4s}},
|q_i-q_j|>\eta\varepsilon^{\frac{1-2s}{4s}}\Big\}
\end{align}
for $s\in(\frac12,1)$ and
\begin{align}
\label{1.12}
Q_{\frac12,\eta}=\Big\{(q_1,q_2,\cdots,q_m)\mid
\frac{1}{\eta}(\log\frac{1}{\varepsilon})^{\frac12}>q_i>\eta(\log\frac{1}{\varepsilon})^{\frac12},
|q_i-q_j|>\eta(\log\frac{1}{\varepsilon})^{\frac12}\Big\}
\end{align}
for $s=\frac12$ respectively. Here $\eta$ denotes a small positive number and with the constants $\alpha_s,\beta_s$ being given later.

We denote one of the global minimal points (if there are more) of $\Xi_s$ in $Q_{s,\eta}$ by
$${\bf{q}}_s=(q_{s,1},q_{s,2},\cdots,q_{s,m}).$$

\noindent{\bf{Remark}}: Since the function $\Xi_s(q_1,q_2,\cdots,q_m)$ is analytic in $Q_{s,\eta}$, therefore, all the global minimal points are discrete.\\

Let us set
\begin{equation}
\label{1.13}
\tau_{\varepsilon}=
\left\{\begin{array}{lll}
\Big(\frac{k}{2s\sin(\frac{\pi}{2s})}\varepsilon^{1-2s}\int_{\mathbb{R}^1}U^2(y)\mathrm{d}y\Big)^{-1},&s\in(\frac12,1),\\
\Big(\frac{k}{\pi}\log\frac{1}{\varepsilon}\int_{\mathbb{R}^1}U^2(y)\mathrm{d}y\Big)^{-1},&s=\frac12.
\end{array}\right.
\end{equation}

\vspace{0.5cm}
Our main result is the following.\\

\begin{theorem}
\label{th1.1}
Let $k=2m,~m\geq1$ be a fixed positive integer. There exists $\varepsilon_k>0$ such that for each $0<\varepsilon<\varepsilon_k,$ problem (\ref{1.1}) admits a solution $(u,v)$ with the following property:
\begin{equation*}
\lim_{\varepsilon\rightarrow0}\Big|\tau_{\varepsilon}^{-1}u_{\varepsilon}(x)-
\sum_{i=1}^m\big(U(x-q_{s,i})+U(x+q_{s,i})\big)\Big|=0,
\end{equation*}
uniformly in $x\in\mathbb{R}^1$, while for the second component $v$, we have
\begin{equation*}
\lim_{\varepsilon\rightarrow0}\Big(\Big|\tau_{\varepsilon}^{-1}v_{\varepsilon}(q_{s,i}+x)-1\Big|+
\Big|\tau_{\varepsilon}^{-1}v_{\varepsilon}(q_{s,i}-x)-1\Big|\Big)=0
\end{equation*}
for any $i$ uniformly in compact sets in $x$, $i=1,2,\cdots,m$. A similar result holds when $ k=2m+1, m\geq 0$.
\end{theorem}

We will only give the details of constructing solutions exhibiting an even number of bumps in the line. The case of odd number of bumps can be treated in a similar manner, and we will discuss the necessary changes in section 8.

\medskip

The method employed in the proof of Theorem \ref{th1.1} consists of a Lyapunov-Schmidt type reduction. Fixing $m$ points which from $Q_{s,\eta}$ with $\eta$ being determined later, an auxiliary problem is solved uniquely, and solutions satisfying the required conditions will be those precisely satisfying a nonlinear system of equations of the form
\begin{equation*}
c_{s,i}(q_1,q_2,\cdots,q_m)=0,~i=1,2,\cdots,m,
\end{equation*}
where for such a class of points the function $c_{s,i}$ satisfy
\begin{equation}
\label{1.14}
c_{s,i}(q_1,q_2,\cdots,q_m)=\frac{\partial}{\partial q_i}\Big[\frac12F_s(|2q_i|)+\sum_{j\neq i}\big(F_s(|q_i-q_j|)
+F_s(|q_i+q_j|)\big)\Big]+\sigma_{s,i},
\end{equation}
where
\begin{align*}
F_s(r)=\left\{\begin{array}{ll}
\alpha_sU(r)+\beta_s\varepsilon^{2s-1}r^{2s-1},~&s\in(\frac12,1),\\
\alpha_{\frac12}U(r)+\beta_{\frac12}\frac{\log r}{\log\frac{1}{\varepsilon}},~&s=\frac12
\end{array}\right.
\end{align*}
and
\begin{align*}
\sigma_{s,i}=\left\{\begin{array}{ll}
o(1)\varepsilon^{s+\frac12-\frac{1}{2s}},~&s\in(\frac12,1),\\
o(1)(-\log\varepsilon)^{-\frac32},~&s=\frac12.
\end{array}\right.
\end{align*}

We can easily find that solutions of the problem $c_{s,i}=0$ are closely related to critical points of the functional defined in (\ref{1.9}) and (\ref{1.10}) for $s\in(\frac12,1)$ and $s=\frac12$ respectively. In the classical case, since the profile $U$ of each bump is exponential decay, in the leading term of $c_{i},$ we only need to consider the neighbor points of $q_i.$ However, in the fractional case,
the profile $U$ has algebraic decay.  As a result, all the points {\em interact with each other strongly} and we have to consider all the interactions in dealing with $c_{s,i}.$ This  is the main difficulty and new feature  in considering (\ref{1.1}).

\medskip

It is clear that the results of Theorem  \ref{th1.1} can be extended without any difficulty to Gierer-Meinhardt system  in which activator and inhibitor have {\em different} diffusion characters
\begin{equation}
\label{1.16}
\left\{\begin{array}{lll}
(-\Delta)^{s_1} u+u-\frac{u^2}{v}=0~&\mathrm{in}~\mathbb{R},\\
(-\Delta)^{s_2} v+\varepsilon^{2s_2}v-u^2=0~&\mathrm{in}~\mathbb{R},\\
\end{array}\right.
\end{equation}
where $ s_1 \in (0, 1]$ and $ s_2 \in [\frac{1}{2}, 1]$. It is also possible to generalize to Gierer-Meinhardt system with more general nonlinearity. We omit the details.

\medskip

The rest of the paper will be devoted to the proof of Theorem \ref{th1.1}. In section 2 we study the fractional Laplacian operator and the behavior of the Green function for $(-\Delta)^s+I$. In section 3 we set up the scheme of proof, in particular we explain why the constant $\tau_{\varepsilon}$ is the right scaling factor to get the desired multi-bump expansion. The program outlined there is carried over the following sections.

\vspace{1cm}
\section{Preliminaries}
In this section, we provide some elementary properties of the operators $(-\Delta)^s+I$. Let $0<s<1$. Various definitions of the fractional Laplacian $(-\Delta)^s\phi$ of a function $\phi$ defined in $\mathbb{R}^n$ are available, depending on its regularity and growth properties.

For $\phi\in H^{2s}(\mathbb{R}^n),$ the standard definition is given via Fourier transform $~\widehat{}~$ . $(-\Delta)^s\phi\in L^2(\mathbb{R}^n)$ is defined by the formula
\begin{equation}
\label{2.1}
|\xi|^{2s}\hat{\phi}(\xi)=\widehat{(-\Delta)^s\phi}.
\end{equation}
When $\phi$ is assumed in addition sufficiently regular, we obtain the direct representation
\begin{equation}
\label{2.2}
(-\Delta)^s\phi(x)=d_{s,n}\int_{\mathbb{R}^n}\frac{\phi(x)-\phi(y)}{|x-y|^{n+2s}}\mathrm{d}y
\end{equation}
for a suitable constant $d_{s,n}$ and the integral is understood in a principal value sense. This integral makes sense directly when $s<\frac12$ and $\phi\in C^{0,\alpha}(\mathbb{R}^n)$ with $\alpha>2s,$ or if $\phi\in C^{1,\alpha}(\mathbb{R}^n)$, $1+\alpha>2s.$ In the latter case, we can desingularize the integral and represent it in the form
\begin{align*}
(-\Delta)^s\phi(x)=d_{s,n}\int_{\mathbb{R}^n}\frac{\phi(x)-\phi(y)-\nabla\phi(x)(x-y)}{|x-y|^{n+2s}}\mathrm{d}y.
\end{align*}
Another useful (local) representation, found by Caffarelli and Silverstre \cite{cs}, is via the following boundary value problem in the half space $\mathbb{R}_+^{n+1}=\{(x,y)\mid x\in\mathbb{R}^n,y>0\}:$
\begin{equation*}
\left\{\begin{array}{ll}
\nabla\cdot(y^{1-2s}\nabla\tilde{\phi})=0~&\mathrm{in}~\mathbb{R}_+^{n+1},\\
\tilde{\phi}(x,0)=\phi(x)&\mathrm{on}~\mathbb{R}^n.
\end{array}\right.
\end{equation*}
Here $\tilde{\phi}$ is the $s-$harmonic extension of $\phi,$ explicitly given as a convolution integral with the $s-$Poisson kernel $p_s(x,y),$
\begin{align*}
\tilde{\phi}(x,y)=\int_{\mathbb{R}^n}p_s(x-z,y)\phi(z)\mathrm{d}z,
\end{align*}
where
\begin{align*}
p_s(x,y)=C_{n,s}\frac{y^{4s-1}}{(|x|^2+|y|^2)^{\frac{n-1+4s}{2}}}
\end{align*}
and $C_{n,s}$ achieves $\int_{\mathbb{R}^n}p(x,y)=1$. Then under suitable regularity, $(-\Delta)^s\phi$ is the Dirichlet-to-Neumann map for this problem, namely
\begin{equation}
\label{2.3}
(-\Delta)^s\phi(x)=\lim_{y\rightarrow0^+}y^{1-2s}\partial_y\tilde{\phi}(x,y).
\end{equation}
Characterizations (\ref{2.1})-(\ref{2.3}) are all equivalent for instance in Schwartz's space of rapidly decreasing smooth functions.

Now let us consider for a number $m>0$ and $g\in L^2(\mathbb{R}^n)$ the equation
\begin{equation*}
(-\Delta)^s\phi+m\phi=g~\mathrm{in}~\mathbb{R}^n.
\end{equation*}
Then in terms of Fourier transform, this problem, for $\phi\in L^2,$ reads
\begin{align*}
(|\xi|^{2s}+m)\hat{\phi}=\hat{g}
\end{align*}
and has a unique solution $\phi\in H^{2s}(\mathbb{R}^n)$ given by the convolution
\begin{equation}
\label{2.4}
\phi(x)=T_m[g]:=\int_{\mathbb{R}^n}G(x-z)g(z)\mathrm{d}z,
\end{equation}
where
$$\widehat{G}(\xi)=\frac{1}{|\xi|^{2s}+m}.$$
Using the characterization (\ref{2.3}) written in weak form, $\phi$ can be characterized by $\phi(x)=\tilde{\phi}(x,0)$ in trace sense, where $\tilde{\phi}\in H$ is the unique solution of
\begin{equation}
\label{2.5}
\int\int_{\mathbb{R}_+^{n+1}}\nabla\tilde{\phi}\nabla\varphi y^{1-2s}+m\int_{\mathbb{R}^n}\phi\varphi=\int_{\mathbb{R}^n}g\varphi,~\mathrm{for~all}~\varphi\in H,
\end{equation}
where $H$ is the Hilbert space of functions $\varphi\in H_{\mathrm{loc}}^1(\mathbb{R}_+^{n+1})$ such that
\begin{equation*}
\|\varphi\|_H^2:=\int\int_{\mathbb{R}_+^{n+1}}|\nabla\varphi|^2 y^{1-2s}+m\int_{\mathbb{R}^n}|\varphi|^2<+\infty,
\end{equation*}
or equivalently the closure of the set of all functions in $C_c^{\infty}(\overline{\mathbb{R}_+^{n+1}})$ under this norm.

A useful fact for our purpose is the equivalence of the representations (\ref{2.4}) and (\ref{2.5}) for $g\in L^2(\mathbb{R}^n).$

\begin{lemma}
\label{le2.1}
Let $g\in L^2(\mathbb{R}^n)$. Then the unique solution $\tilde{\phi}\in H$ of problem (\ref{2.5}) is given by the $s-$harmonic extension of the function $\phi=T_m[g]=G*g.$
\end{lemma}

For a proof, one can see Lemma 2.1 in \cite{ddw}. Let us recall the main properties of the fundamental solution $G(x)$ in the representation (\ref{2.4}), which are stated for instance in \cite{fqt} and \cite{fls}.

We have that $G$ is radially symmetric and positive, $G\in C^{\infty}(\mathbb{R}^n\setminus\{0\})$ satisfying
\begin{itemize}
  \item $$|G(x)|+|x||\nabla G(x)|\leq \frac{C}{|x|^{n-2s}}~\mathrm{for~all}~|x|\leq1,$$
  \item $$\lim_{|x|\rightarrow\infty}G(x)|x|^{n+2s}=\gamma>0,$$
  \item $$|x||\nabla G(x)|\leq\frac{C}{|x|^{n+2s}}~\mathrm{for~all}~|x|\geq1.$$
\end{itemize}

In next lemma, we get a more specific behavior of the Green function $G(x)$ around $x=0$ for $n=1.$

\begin{lemma}
\label{le2.2}
Let $G(x)$ be the Green function of the following equation
\begin{equation}
\label{2.6}
(-\Delta)^sG(x)+G(x)=\delta(x).
\end{equation}
Then, we have
\begin{align}
\label{2.7}
G(x)=\left\{\begin{array}{ll}
\mathfrak{a}_0+\mathfrak{a}_1|x|^{2s-1}+O(|x|^{\min\{2,4s-1\}})~\mathrm{as}~|x|\rightarrow0,~&\mathrm{if}~s\in(\frac12,1),\\
-\frac{1}{\pi}\log|x|+\mathfrak{a}_2+O(|x|)~\mathrm{as}~|x|\rightarrow0,~&\mathrm{if}~s=\frac12,
\end{array}\right.
\end{align}
where $\mathfrak{a}_0,\mathfrak{a}_1,\mathfrak{a}_2$ will be given in the proof.
\end{lemma}

\noindent{\bf{Remark}}: $\mathfrak{a}_1$ is negative.

\begin{proof}
By using fourier transform, we can write the equation (\ref{2.6}) as
\begin{align}
\label{2.8}
(|\xi|^{2s}+1)\widehat{G}(\xi)=1.
\end{align}
Therefore, we have
\begin{align}
\label{2.9}
G(x)=\frac{1}{2\pi}\int_{-\infty}^{\infty}\frac{e^{ix\xi}}{1+|\xi|^{2s}}\mathrm{d}\xi
=\frac{1}{\pi}\int_{0}^{\infty}\frac{\cos(x\xi)}{1+|\xi|^{2s}}\mathrm{d}\xi.
\end{align}
We note $G(x)$ is an even function, in the following, we only need to consider the behavior of $G(x)$ when $x\rightarrow0^+$. We divide our discussion into two parts.

If $s\in(\frac12,1),$ we have
\begin{align}
\label{2.10}
\frac{1}{\pi}\int_{0}^{\infty}\frac{\cos(x\xi)}{1+|\xi|^{2s}}\mathrm{d}\xi=~&
\frac{1}{\pi}\Big(\int_0^{\infty}\frac{1}{1+|\xi|^{2s}}\mathrm{d}\xi
-2\int_0^{\infty}\frac{(\sin\frac{x\xi}{2})^2}{1+|\xi|^{2s}}\mathrm{d}\xi\Big)\nonumber\\
=~&\frac{1}{2s\sin(\frac{\pi}{2s})}-\frac{2}{\pi}\int_0^{\infty}\frac{(\sin\frac{x\xi}{2})^2}{1+|\xi|^{2s}}\mathrm{d}\xi
\nonumber\\
=~&\frac{1}{2s\sin(\frac{\pi}{2s})}
-\frac{2}{\pi}x^{2s-1}\int_0^{\infty}\frac{(\sin\frac{t}{2})^2}{t^{2s}+x^{2s}}\mathrm{d}t\nonumber\\
=~&\frac{1}{2s\sin(\frac{\pi}{2s})}-\frac{2}{\pi}x^{2s-1}\int_0^{\infty}\Big(\frac{(\sin\frac{t}{2})^2}{t^{2s}}
-\frac{x^{2s}(\sin\frac{t}{2})^2}{t^{2s}(t^{2s}+x^{2s})}\Big)\mathrm{d}t\nonumber\\
=~&\mathfrak{a}_0+\mathfrak{a}_1x^{2s-1}+
\frac{2}{\pi}x^{4s-1}\int_0^{\infty}\frac{(\sin\frac{t}{2})^2}{t^{2s}(t^{2s}+x^{2s})}\mathrm{d}t.
\end{align}
where
$$\mathfrak{a}_0=\frac{1}{2s\sin(\frac{\pi}{2s})},~
\mathfrak{a}_1=-\frac{2}{\pi}\int_0^{\infty}\frac{(\sin\frac{t}{2})^2}{t^{2s}}\mathrm{d}t=-\frac{2}{\pi}s\Gamma(-2s)\sin(\pi s)<0.$$
We can write the last term on the right hand side of (\ref{2.10}) as
\begin{equation}
\label{2.11}
\int_0^{\infty}\frac{x^{4s-1}(\sin\frac{t}{2})^2}{t^{2s}(t^{2s}+x^{2s})}\mathrm{d}t
=\Big(\int_{0}^{x}+\int_x^1+\int_1^{\infty}\Big)
\frac{x^{4s-1}(\sin\frac{t}{2})^2}{t^{2s}(t^{2s}+x^{2s})}\mathrm{d}t.
\end{equation}
For each term on the right hand side of (\ref{2.11}), we have
\begin{align*}
&\int_{0}^{x}\frac{x^{4s-1}(\sin\frac{t}{2})^2}{t^{2s}(t^{2s}+x^{2s})}\mathrm{d}t
\leq Cx^{2s-1}\int_0^x\frac{t^2}{t^{2s}}\mathrm{d}t=O(x^2),
\end{align*}
\begin{align*}
&\int_{x}^{1}\frac{x^{4s-1}(\sin\frac{t}{2})^2}{t^{2s}(t^{2s}+x^{2s})}\mathrm{d}t
\leq Cx^{4s-1}\int_x^1\frac{t^2}{t^{4s}}\mathrm{d}t=O(x^2+x^{4s-1}),
\end{align*}
\begin{align*}
&\int_{1}^{\infty}\frac{x^{4s-1}(\sin\frac{t}{2})^2}{t^{2s}(t^{2s}+x^{2s})}\mathrm{d}t
\leq Cx^{4s-1}\int_1^{\infty}\frac{(\sin\frac t2)^2}{t^{4s}}\mathrm{d}t=O(x^{4s-1}).
\end{align*}
Hence, we have
$$\int_0^{\infty}\frac{x^{4s-1}(\sin\frac{t}{2})^2}{t^{2s}(t^{2s}+x^{2s})}\mathrm{d}t=O(x^{\min\{2,4s-1\}}),$$
as a result, we get the first one in (\ref{2.7}).

When $s=\frac12,$ we have
\begin{align*}
\int_0^{\infty}\frac{\cos(x\xi)}{1+\xi}\mathrm{d}\xi=~&
\int_0^{\frac{\pi}{2x}}\frac{\cos(x\xi)}{1+\xi}\mathrm{d}\xi
+\int_{\frac{\pi}{2x}}^{\infty}\frac{\cos(x\xi)}{1+\xi}\mathrm{d}\xi\\
=~&\int_0^{\frac{\pi}{2x}}\Big(\frac{1}{1+\xi}-\frac{2(\sin \frac{x\xi}{2})^2}{1+\xi}\Big)\mathrm{d}\xi
+\int_{\frac{\pi}{2x}}^{\infty}\frac{\cos(x\xi)}{1+\xi}\mathrm{d}\xi\\
=~&\log(1+\frac{\pi}{2x})-2\int_0^{\frac{\pi}{2}}\frac{(\sin\frac t2)^2}{x+t}\mathrm{d}t
+\int_{\frac{\pi}{2}}^{\infty}\frac{\cos t}{x+t}\mathrm{d}t\\
=~&-\log x+\log\frac{\pi}{2}-2\int_0^{\frac{\pi}{2}}\frac{(\sin\frac t2)^2}{t}\mathrm{d}t+\int_{\frac{\pi}{2}}^{\infty}\frac{\cos t}{t}\mathrm{d}t\\
&+2x\int_0^{\frac{\pi}{2}}\frac{(\sin\frac t2)^2}{t(x+t)}\mathrm{d}t-x\int_{\frac{\pi}{2}}^{\infty}\frac{\cos t}{t(x+t)}\mathrm{d}t+O(x).
\end{align*}
It is easy to see,
\begin{align*}
\int_0^{\frac{\pi}{2}}\frac{(\sin\frac t2)^2}{t(x+t)}\mathrm{d}t,~\int_{\frac{\pi}{2}}^{\infty}\frac{\cos t}{t(x+t)}\mathrm{d}t=O(1).
\end{align*}
Therefore, for $s=\frac12,$ we obtain
\begin{equation}
\label{2.12}
G(x)=-\frac{1}{\pi}\log x+\mathfrak{a}_2+O(x)~\mathrm{as}~x\rightarrow0^+,
\end{equation}
where
$$\mathfrak{a}_2=\frac{1}{\pi}\Big(\log\frac{\pi}{2}-2\int_0^{\frac{\pi}{2}}\frac{(\sin\frac t2)^2}{t}\mathrm{d}t+\int_{\frac{\pi}{2}}^{\infty}\frac{\cos t}{t}\mathrm{d}t\Big).$$
\end{proof}

\vspace{1cm}
\section{The scheme of the proof}
Our strategy of the proof of the main results is based on the idea of solving the second equation in (\ref{1.1}) for $v$ and then working with a nonlocal elliptic PDE rather than directly with the system. It is convenient to do this by replacing first $u$ by $\tau_{\varepsilon}u$ and $v$ by $\tau_{\varepsilon}v$, which transforms (\ref{1.1}) into the problem
\begin{equation}
\label{3.1}
\left\{\begin{array}{ll}
(-\Delta)^su+u-\frac{u^2}{v}=0~&\mathrm{in}~\mathbb{R},\\
(-\Delta)^sv+\varepsilon^{2s}v-\tau_{\varepsilon}u^2=0~&\mathrm{in}~\mathbb{R},\\
u,v>0,\quad u,v\rightarrow0,&\mathrm{as}~|x|\rightarrow\infty,
\end{array}\right.
\end{equation}
with the choice of the parameter $\tau_{\varepsilon}$ as in (\ref{1.13}),
\begin{equation}
\label{3.2}
\tau_{\varepsilon}=
\left\{\begin{array}{lll}
\Big(\frac{k}{2s\sin(\frac{\pi}{2s})}\varepsilon^{1-2s}\int_{\mathbb{R}^1}U^2(y)\mathrm{d}y\Big)^{-1},&s\in(\frac12,1),\\
\Big(\frac{k}{\pi}\log\frac{1}{\varepsilon}\int_{\mathbb{R}^1}U^2(y)\mathrm{d}y\Big)^{-1},&s=\frac12.
\end{array}\right.
\end{equation}
Then, we have
$$u\sim\sum_{i=1}^kU(x-q_i),\quad v\sim1,$$
i.e., the height of the bumps near the $q_i$ remains bounded as $\varepsilon\rightarrow0.$

In the sequel, by $T(h)$ we denote the unique solution of the equation
\begin{align*}
&(-\Delta)^sv+\varepsilon^{2s}v=\tau_{\varepsilon}h~\mathrm{in}~\mathbb{R},v(x)\rightarrow0~\mathrm{as}~
|x|\rightarrow+\infty,
\end{align*}
for $h\in L^2(\mathbb{R})$, namely $T=\tau_{\varepsilon}\big((-\Delta)^{s}+\varepsilon^{2s}\big)^{-1}$. Solving the second equation for $v$ in (\ref{3.1}) we get $v=T(u^2)$, which tends to the nonlocal equation
\begin{equation}
\label{3.3}
(-\Delta)^su+u-\frac{u^2}{T(u^2)}=0.
\end{equation}
We consider points $q_1,q_2,\cdots,q_{k},~k=2m$ in $\mathbb{R}$ which are the candidates for the location of spikes.
These points $(q_1,q_2,\cdots,q_k)$ are restricted in the following set
\begin{align}
\label{3.4}
(q_1,q_2,\cdots,q_k)\in \Lambda_s=\Big\{&(q_1,q_2,\cdots,q_k)\mid q_i=-q_{k+1-i},~q_1>q_2>\cdots>q_k,\nonumber\\
&(q_1,q_2,\cdots,q_m)\in Q_{s,\eta}\Big\}.
\end{align}
Let us write
$$W(x)=\sum_{i=1}^kU(x-q_i).$$
We look for a solution to (\ref{3.3}) in the form $u=W+\phi$, where $\phi$ is a lower order term. Then, formally, we have
$$T(u^2)=T(W^2)+2T(W\phi)+h.o.t.,$$
where $h.o.t.$ corresponds to the higher order terms. We denote $V=T(W^2).$ By $G_{\varepsilon}(|y|)$ we denote the fundamental solution to $(-\Delta)^s+\varepsilon^{2s}$ in $\mathbb{R}$. We can write
\begin{align*}
T(W^2)=\tau_{\varepsilon}\int W^2(y)G_{\varepsilon}(|x-y|)dy\sim \tau_{\varepsilon}\sum_{i=1}^k\int U^2(x-q_i)G_{\varepsilon}(|x-y|)\mathrm{d}y,
\end{align*}
where the integration extends over all $\mathbb{R}$. For $|x|=o(\varepsilon^{-1})$ we have
\begin{align*}
G_{\varepsilon}(|x|)=
\left\{\begin{array}{ll}
\alpha_0\varepsilon^{1-2s}+\alpha_1|x|^{2s-1}+O(\varepsilon^{3-2s}|x|^2+\varepsilon^{2s}|x|^{4s-1}),
&~\mathrm{if}~s\in(\frac12,1),\\
-\frac{1}{\pi}\log(\varepsilon|x|)+O(1),&~\mathrm{if}~s=\frac12.
\end{array}\right.
\end{align*}
Using this, and the definition of $\tau_{\varepsilon}$ we get that near the $q_i$'s,
\begin{equation*}
V(x)=1+h.o.t..
\end{equation*}
Arguing similarly we get
\begin{align*}
T(W\phi)=\omega\int W\phi+h.o.t.,~\omega=\frac{1}{k\int U^2}.
\end{align*}
Then
\begin{align*}
\frac{u^2}{v}=\frac{W^2+2W\phi+h.o.t.}{V+2T(W\phi)+h.o.t.}=\frac{W^2}{V}+2W\phi-2\omega W^2\int W\phi+h.o.t..
\end{align*}
Substituting all this into (\ref{3.3}) we obtain the equation for $\phi$
\begin{equation}
\label{3.5}
(-\Delta)^s\phi+(1-2W)\phi+2\omega W^2\int W\phi=S(W)+N(\phi),
\end{equation}
where $S(W)=-(-\Delta)^sW-W+\frac{W^2}{V}$ and $N(\phi)$ is defined by
\begin{equation*}
N(\phi)=\frac{(W+\phi)^2}{T((W+\phi)^2)}-\frac{W^2}{V}-2W\phi+2\omega W^2\int W\phi,
\end{equation*}
represents higher order terms in $\phi$.

Thus we have reduced the problem of finding solutions to (\ref{3.1}) to the problem of solving (\ref{3.5}) for $\phi$. We set $\frac{\partial W}{\partial q_j}=Z_j$. Rather than solving problem (\ref{3.5}), we consider first the following auxiliary problem: given points $q_j$, find a function $\phi$ such that for certain constants $c_i$ the following equation is satisfied
\begin{align}
\label{3.6}
L\phi=S(W)+N(\phi)+\sum_jc_jZ_j,~\langle \phi,Z_j\rangle=0,~j=1,2,\cdots,k,
\end{align}
where
$$L\phi=(-\Delta)^{s}\phi+(1-2W)\phi+2\omega W^2\int W\phi$$
and $\langle\cdot,\cdot\rangle$ denotes the $L^2$ inner product.

We will prove in section 4 that this problem is uniquely solvable within a class of small functions $\phi$ for all points $(q_1,q_2,\cdots,q_k)$ satisfying constraints (\ref{3.4}). Besides, the resulting constants $c_i(q_1,q_2,\cdots,q_k)$ admit the expansion (\ref{1.14}) . We will of course get a solution of the full problem whenever the points $q_i$ are adjusted in such a way that all of $c_i$'s vanish. We show the existence of such points in section 8, where the main result is finally established. In remainder of the paper we rigorously carry out the program outlined above. In particular, we will need to understand the invertibility properties of the linear operator $L$ first. We will do it in the next section.

\vspace{1cm}
\section{The linear operator}
\begin{proposition}
\label{pr4.1}
Let $U$ be the unique, positive, radially symmetric solution to
\begin{equation}
\label{4.1}
(-\Delta)^sU+U-U^2=0.
\end{equation}
\begin{itemize}
  \item [(a)] There exists a positive constant $d_s$ depending on $s$ only such that, as $|x|\rightarrow\infty,$ the following formula holds
  \begin{equation*}
  U(|x|)=\frac{\mathfrak{b}_s}{|x|^{1+2s}}(1+o(1)).
  \end{equation*}
  Moreover, $U'(x)<0$ for $x>0$ and
  \begin{equation*}
  U'(x)=-\frac{(1+2s)\mathfrak{b}_s}{x^{2+2s}}(1+o(1))~\mathrm{as}~x\rightarrow+\infty.
  \end{equation*}
  \item [(b)] Let $L_0=(-\Delta)^s+(1-2U)\mathrm{id}$. Then we have
  \begin{equation*}
  \mathrm{Ker}(L_0)=\mathrm{span}\Big\{\frac{\partial U}{\partial x}\Big\}.
  \end{equation*}
  \item [(c)] Let $L$ be the operator defined in (\ref{3.6}) and let
  \begin{equation*}
  L^*=(-\Delta)^s\phi+(1-2W)\phi+2\omega W\int W^2\phi
  \end{equation*}
  be its formal disjoint. If we denote
  \begin{equation*}
  Z_j=\frac{\partial W}{\partial q_j},
  \end{equation*}
then, for all $j=1,2,\cdots,k$, we have
\begin{align*}
&LZ_j=O(\min_{i\neq j}|q_i-q_j|^{-(2s+1)})W(x)^2+O(\min_{i\neq j}|q_i-q_j|^{-(2s+2)})W(x),\\
&\quad\quad\quad\quad\quad\quad\quad L^*Z_j=O(\min_{i\neq j}|q_i-q_j|^{-(2s+2)})W(x).
\end{align*}
\end{itemize}
\end{proposition}

\begin{proof}
The proof of part (a) and (b) is given in \cite[Proposition 1.1 and Theorem 2.3]{fl}. In fact, it was proven in \cite{at} that (apart from translations) the function $U(x)=\frac{2}{1+x^2}$ is the only positive solution of (\ref{4.1}) in $H^{\frac12}(\mathbb{R})$ when $s=\frac12$. Part (c) is a consequence of direct computations.
\end{proof}

We shall carry out the analysis of the linear operator $L$ in a framework of weighted $L^{\infty}$ spaces. For this purpose we consider the following norm for a function defined in $\mathbb{R}$: given points $q_1,q_2,\cdots,q_k$ we define
\begin{equation}
\label{4.2}
\|\phi\|_*=\|\rho(x)^{-1}\phi\|_{L^{\infty}(\mathbb{R})},
\end{equation}
where
$$\rho(x)=\sum_{j=1}^k\frac{1}{(1+|x-q_j|)^{\mu}},~\frac12<\mu\leq1+2s.$$

We first consider a problem that will later give rise to the finite-dimensional reduction. Given a function $h$, $\|h\|_*<\infty$ find a $\phi$ and constants $c_j,~j=1,2,\cdots,k$ such that one has
\begin{equation}
\label{4.3}
\left\{\begin{array}{ll}
L\phi=h+\sum_jc_jZ_j~\mathrm{in}~\mathbb{R},\\
\phi(x)\rightarrow0~\mathrm{as}~|x|\rightarrow\infty,\\
\langle \phi, Z_j\rangle=0~\mathrm{for}~j=1,2,\cdots,k.
\end{array}\right.
\end{equation}
By $\mathbf{c}$ we will denote a vector with components $c_j.$

We refer to a pair $(\phi,\mathbf{c})$ as a solution to (\ref{4.3}). We have the following existence result for (\ref{4.3}).
\begin{theorem}
\label{th4.1}
There exists positive numbers $R$ and $C$ such that, for any points $q_j$ satisfying $|q_i-q_j|>R$ for all $i\neq j$, and $h$ with $\|h\|_*<\infty$, problem (\ref{4.3}) has a unique solution $\phi=\mathcal{T}(h)$ and $\mathbf{c}=\mathbf{c}(h)$. Moreover,
\begin{equation}
\label{4.4}
\|\mathcal{T}(h)\|_*\leq C\|h\|_{*}.
\end{equation}
\end{theorem}

The proof of Theorem \ref{th4.1} relies heavily on the following lemma.
\begin{lemma}
\label{le4.1}
Assume that $q_j^n,j=1,2,\cdots,k,$ are such that $\min_{i\neq j}|q_i^n-q_j^n|\rightarrow\infty$, $\|h_n\|_*\rightarrow0$ and that $\phi_n$ solves
\begin{align*}
&L\phi_n=h_n+\sum_jc_j^nZ_j~\mathrm{in}~\mathbb{R},\\
&\phi_n(x)\rightarrow0,\quad\mathrm{as}~|x|\rightarrow\infty,\\
&\langle\phi_n,Z_j\rangle=0~\mathrm{for}~j=1,2,\cdots,k.
\end{align*}
Then $\|\phi_n\|_*\rightarrow0.$
\end{lemma}

\begin{proof}
We will argue by contradiction. Without loss of generality we can assume that $\|\phi_n\|_*=1.$ Our first observation is that $c_{j}^n\rightarrow0.$ Indeed, multiplying the equation by $Z_j$ and integrating by parts we get
\begin{equation*}
\langle\phi_n,L^*Z_j\rangle=c_j^n\int Z_j^2+\sum_{i\neq j}c_i^n\langle Z_i,Z_j\rangle+\langle h_n,Z_j\rangle.
\end{equation*}
Using Proposition \ref{pr4.1}, by rather standard calculations, it follows that $c_j^n\rightarrow0$ as $n\rightarrow\infty.$

Our next goal is to show
\begin{align*}
\int W\phi_n\rightarrow 0~\mathrm{as}~n\rightarrow\infty.
\end{align*}
To this end consider the test function
$$Z=\frac1sx\cdot\nabla W+2W$$
and
$$L_0=(-\Delta)^s+(1-2U)\mathrm{id}.$$
We claim that
$$L_0Z=-2W+o(1).$$
Indeed if we set $u_{\lambda}(x)=\lambda^{2s}W(\lambda x),$ then
\begin{equation*}
-(-\Delta)^{s} u_{\lambda}(x)=\lambda^{2s}u_{\lambda}-u_{\lambda}^2+\sum_{i\neq j}\lambda^{4s}U(\lambda x-\xi_i)U(\lambda x-\xi_j).
\end{equation*}
Since $Z=\frac{1}{s}\frac{\partial u_{\lambda}}{\partial\lambda}\mid_{\lambda=1}$ the claim follows now from Proposition \ref{pr4.1} and the above.

Decompose $\phi_n=a_nW+\psi_n$ where $\langle W,\psi_n\rangle=0.$ Then $L\psi_n=L_0\psi_n$ and we have
\begin{equation*}
o(1)=\langle L\phi_n,Z\rangle=a_n\langle LW,Z\rangle+\langle L_0\psi_n,Z\rangle.
\end{equation*}
But
\begin{equation*}
\langle L_0\psi_n,Z\rangle=\langle \psi_n,L_0Z\rangle=-2\langle W,\psi_n\rangle+o(1)=o(1)
\end{equation*}
and
\begin{equation*}
\langle LW,Z\rangle=\langle W^2,Z\rangle+o(1)=\frac{1}{3s}\int x\nabla W^3+2\int W^3+o(1)=(2-\frac{2}{3s})\int W^3+o(1).
\end{equation*}
It follows that $a_n\rightarrow0,$ and $\langle W,\phi_n\rangle=o(1)$. Going back to the equation satisfied by $\phi_n$, we see that
\begin{equation*}
(-\Delta)^s\phi_n+(1-2W)\phi_n=o(1)\Big(W^2+\sum_jZ_j\Big)+h_n\equiv g_n+h_n,
\end{equation*}
with $\|g_n\|_*\rightarrow0$. By applying the Lemma 4.2 in \cite{ddw}, we get
\begin{equation*}
\|\phi_n\|_*\leq C(\|g_n\|_*+\|h_n\|_*)=o(1)
\end{equation*}
which contradicts to the assumption $\|\phi_n\|_*=1.$ Hence we finish the proof of this lemma.
\end{proof}

\noindent {\bf{Remark:}} In the statement of Lemma 4.2 in \cite{ddw}, it requires that $\mu\in(\frac12,1+2s).$ However, we can go through the proof to get the same conclusion for the case $\mu=1+2s.$\\

Next we construct a solution to problem (\ref{4.3}). To do so, we consider the following auxiliary problem at first
\begin{equation}
\label{4.5}
\left\{\begin{array}{ll}
(-\Delta)^s\phi+\phi=g+\sum_{i=1}^kc_iZ_i,\\
\int_{\mathbb{R}^1}\phi Z_i=0~\mathrm{for~all}~i.
\end{array}\right.
\end{equation}

\begin{lemma}
\label{le4.2}
For each $g$ with $\|g\|_{*}<+\infty$, there exists a unique solution of problem (\ref{4.5}), $\phi=:\mathcal{A}[g]\in H^{2s}(\mathbb{R})$. This solution satisfies
\begin{align*}
\|\mathcal{A}[g]\|_{*}\leq C\|g\|_{*}.
\end{align*}
\end{lemma}

For the proof of Lemma \ref{le4.2}, one can see \cite[Lemma 4.3]{ddw}.\\

\noindent{\em Proof of Theorem \ref{th4.1}:}
Let us solve problem (\ref{4.3}). Let $\mathcal{Y}$ be the Banach space
\begin{equation}
\label{4.6}
\mathcal{Y}:=\{\phi\in L^{\infty}(\mathbb{R}^1)\mid \|\phi\|_{\mathcal{Y}}:=\|\phi\|_{*}<\infty\}.
\end{equation}
Let $\mathcal{A}$ be the operator defined in Lemma \ref{le4.2}. Then we have a solution to (\ref{4.3}) if we can solve
\begin{equation}
\label{4.7}
\phi-\mathcal{A}[2W\phi]+2\mathcal{A}[\omega W^2\int W\phi]=\mathcal{A}[g],~\phi\in\mathcal{Y}.
\end{equation}
We claim that
\begin{align*}
\mathcal{B}[\phi]:=\mathcal{B}_1[\phi]+\mathcal{B}_2[\phi]=\mathcal{A}[2W\phi]+\mathcal{A}[-2\omega W^2\int W\phi]
\end{align*}
defines a compact operator in $\mathcal{Y}$.
At first, we show that $\mathcal{B}_2$ defines a compact operator in $\mathcal{Y}$. We begin to prove that $\mathcal{C}[\phi]=\omega W^2\int W\phi$ is a compact operator from $\mathcal{Y}$ to $\mathcal{Y}.$ Let us assume that $\phi_n$ is a bounded sequence in $\mathcal{Y}.$ By direct computations, we get $I(\phi_n)=\int W\phi_n$ is contained in a bounded set in $\mathbb{R}^1.$ Therefore, we can pick out subsequence of $I(\phi_n)$ converge, combined with the fact $W(x)\leq C\rho(x),$ we get $\mathcal{C}$ is a compact operator from $\mathcal{Y}$ to $\mathcal{Y}.$ By Lemma \ref{le4.2}, we have $\mathcal{A}$ is a continuous operator from $\mathcal{Y}$ to $\mathcal{Y}.$ Hence, we get $\mathcal{B}_2$ is compact.

Then, we come to consider the operator $\mathcal{B}_1$. We denote $f_n=\mathcal{A}[2W\phi_n]$ and claim that
\begin{align}
\label{4.8}
\sup_{x\neq y}\frac{|f_n(x)-f_n(y)|}{|x-y|^{\alpha}}\leq C\|\phi_n\|_{\infty},~\alpha=\min\{1,2s\}.
\end{align}
By Green representation, we have
\begin{align*}
f_n=\mathcal{A}[2W\phi_n]=2\int G(x,z)W\phi_n\mathrm{d}z,
\end{align*}
Since
\begin{equation*}
\|\mathcal{A}[2W\phi_n]\|_{L^{\infty}}=\|2\int G(x,z)(W\phi_n)(z)\mathrm{d}z\|_{L^{\infty}}\leq C\|\phi_n\|_{\infty},
\end{equation*}
we can obtain
\begin{align*}
\sup_{x\neq y}\frac{|f_n(x)-f_n(y)|}{|x-y|^{\alpha}}\leq C\|\phi_n\|_{\infty}
\end{align*}
for $|x-y|\geq\frac13.$ For $|x-y|<\frac13,$ we have
\begin{align*}
|f_n(x)-f_n(y)|\leq C\|\phi_n\|_{\infty}\int_{\mathbb{R}}|G(x-z)-G(y-z)|\mathrm{d}z.
\end{align*}
Now, we decompose
\begin{align*}
\int_{\mathbb{R}}|G(x-z)-G(y-z)|\mathrm{d}z=&
\int_{\mathbb{R}}|G(z+y-x)-G(z)|\mathrm{d}z\\
=&\int_{|z|>3|y-x|}|G(z+y-x)-G(z)|\mathrm{d}z\\
&+\int_{|z|\leq3|y-x|}|G(z+y-x)-G(z)|\mathrm{d}z.
\end{align*}
We have
\begin{align*}
\int_{|z|>3|y-x|}|G(z+y-x)-G(z)|\mathrm{d}z\leq
\int_0^1\mathrm{d}t\int_{|z|>3|y-x|}|\nabla G(z+t(y-x))|\mathrm{d}z|y-x|,
\end{align*}
and since $3|y-x|<1$,
\begin{align*}
\int_{|z|>3|y-x|}|\nabla G(z+t(y-x))|\mathrm{d}z&\leq C(1+\int_{1>|z|>3|y-x|}\frac{\mathrm{d}z}{|z|^{2-2s}})\\
&\leq\left\{\begin{array}{ll}
C(1+|y-x|^{2s-1}),~&s\in(\frac12,1),\\
C(1+\log|y-x|),~&s=\frac12.
\end{array}\right.
\end{align*}
On the other hand
\begin{align*}
\int_{|z|<3|y-x|}|G(z+y-x)-G(z)|\mathrm{d}z\leq~&2\int_{|z|<4|y-x|}|G(z)|\mathrm{d}z\\
\leq&
\left\{\begin{array}{ll}
C|x-y|,~&s\in(\frac12,1),\\
C|x-y|\log|x-y|,~&s=\frac12.
\end{array}\right.
\end{align*}
Hence, we get
\begin{align*}
\sup_{x\neq y}\frac{|f_n(x)-f_n(y)|}{|x-y|^{\alpha}}\leq C\|\phi_n\|_{\infty}.
\end{align*}
Therefore, (\ref{4.8}) is proved.

Using the Arzela-Ascoli
theorem, we can get a subsequence of $f_n$ which we relabel the same, that converges uniformly on compact sets to a continuous function $f$. Outside such a compact set, by using the property of the Green function and standard potential analysis, we have
\begin{align*}
f_n=\mathcal{A}[2W\phi_n]=\frac{\gamma_n}{|x|^{1+2s}}(1+o(1))~\mathrm{as}~|x|\rightarrow\infty,
\end{align*}
where $\gamma_n=2\gamma\int_{\mathbb{R}}W\phi_n$. Based on the previous subsequence we have chosen, we can further choose a subsequence such that $I(\phi_n)=\int_{\mathbb{R}}W\phi_n$ converges, combined with the fact $\rho(x)^{-1}|x|^{-2s-1}\rightarrow k$ as $x\rightarrow\infty$, we can get a subsequence of $f_n$ converges in $\mathcal{Y}$ outside a compact set. This implies that $\mathcal{B}_1$ is compact. Therefore we prove the claim that $\mathcal{B}$ defines a compact operator in $\mathcal{Y}.$

Finally, a priori estimate tells us that for $g=0$, equation (\ref{4.3}) admits only the trivial solution. The desired result of Theorem \ref{th4.1} follows at once from Fredholm's alternative.\quad\quad\quad\quad\quad\quad\quad\quad\quad\quad\quad\quad\quad\quad\quad\quad
\quad\quad\quad\quad\quad\quad\quad\quad\quad\quad\quad\quad\quad\quad$\square$

\vspace{1cm}
\section{Basic estimates}

In this section we calculate basic estimates, including error estimates.
Our first task  is to analyze  the solution $V$ of the problem
\begin{equation*}
(-\Delta)^sV+\varepsilon^{2s}V=\tau_{\varepsilon}\Big[\sum_{i=1}^kU(|x-q_i|)\Big]^2,\quad V(x)\rightarrow0~\mathrm{as}~|x|\rightarrow\infty,
\end{equation*}
where $\tau_{\varepsilon}$ is given in (\ref{3.2}). Denote by $Z_0$ the solution of
\begin{equation*}
(-\Delta)^sZ_0+\varepsilon^{2s} Z_0=U(|x|)^2,\quad Z_0(x)\rightarrow0~\mathrm{as}~|x|\rightarrow\infty,
\end{equation*}
and by $\theta_{ij}(x),~i\neq j$, that of
\begin{equation}
\label{5.1}
(-\Delta)^s\theta_{ij}+\varepsilon^{2s}\theta_{ij}=U(|x-q_i|)U(|x-q_j|),\quad
\theta_{ij}(x)\rightarrow0~\mathrm{as}~|x|\rightarrow\infty.
\end{equation}
Then we have
\begin{equation*}
V(x)=\tau_{\varepsilon}\sum_{i=1}^kZ_0(|x-q_i|)+\tau_{\varepsilon}\sum_{i\neq j}\theta_{ij}(x).
\end{equation*}
We will now study $Z_0(|x|)$ in the range
$$x\in I_s:=[-M\varepsilon^{\frac{1-2s}{4s}},M\varepsilon^{\frac{1-2s}{4s}}]$$
for $s\in(\frac12,1)$ and in the range
$$x\in I_{\frac12}:=[-M(-\log\varepsilon)^{\frac12},M(-\log\varepsilon)^{\frac12}]$$
for $s=\frac12.$ Here $M=\frac{100}{\eta}.$

By Green representation, we have $Z_0(|x|)=\int G_{\varepsilon}(|x-y|)U(|y|)^2\mathrm{d}y.$ We can expand $Z_0(|x|)$ as
\begin{equation}
\label{5.2}
Z_0(|x|)=\left\{\begin{array}{ll}
\int\big(\mathfrak{a}_0\varepsilon^{1-2s}+\mathfrak{a}_1|x-y|^{2s-1}\big)U(|y|)^2+O(\varepsilon^{\min\{2+\frac{1}{2s}-2s,
\frac32-\frac{1}{4s}\}}),~&s\in(\frac12,1),\\
-\frac{1}{\pi}\int\log(\varepsilon|x-y|)U(|y|)^2+\mathfrak{a}_2\int U^2+O(\varepsilon(-\log\varepsilon)^{\frac12}),~&s=\frac12.
\end{array}\right.
\end{equation}

Let us consider the quantity $H_s(x)=\mathfrak{a}_1\int|x-y|^{2s-1}U(|y|)^2$ for $s\in(\frac12,1)$, by noting that $U(|y|)$ is an even function, we can easily obtain
\begin{equation*}
H_s(-x)=\mathfrak{a}_1\int|-x-y|^{2s-1}U(|y|)^2=\mathfrak{a}_1\int|-x+y|^{2s-1}U(|y|)^2=H_s(x).
\end{equation*}
Furthermore, we can write
\begin{equation*}
H_s(|x|)=\mathfrak{a}_1|x|^{2s-1}\int U(|y|)^2+f_s(|x|),
\end{equation*}
where $f_s$ and its first derivative are uniformly bounded. Similarly, we can get
\begin{equation*}
H_{\frac12}(|x|)=-\frac{1}{\pi}\log|x|\int U(|y|)^2+f_{\frac12}(|x|),
\end{equation*}
where $f_{\frac12}$ and its first derivative are uniformly bounded.

Let us now consider the function $\theta_{ij}$ given in (\ref{5.1}). Since $\theta_{ij}$ can be represented as
\begin{equation*}
\theta_{ij}=\int G_{\varepsilon}(|x-y|)U(y-q_i)U(y-q_j).
\end{equation*}
At first, we consider the case for $s\in(\frac12,1)$. For $x\in I_s,$ the following expansion holds,
\begin{align*}
\theta_{ij}(x)=&\mathfrak{a}_0\varepsilon^{1-2s}\int U(y-q_i)U(y-q_j)+\mathfrak{a}_1\int |x-y|^{2s-1}U(y-q_i)U(y-q_j)\\
&+O(\varepsilon^{\frac12})\int U(y-q_i)U(y-q_j).
\end{align*}
Using Proposition \ref{pr4.1} one can show that
\begin{align*}
\int |x-y|^{2s-1}U(y-q_i)U(y-q_j)=O(\varepsilon^{\frac{2s-1}{2s}}),
\end{align*}
uniformly in $x\in I_s,$ a similar estimate holds for the derivative of the above expression with respect to $x$. Let us set
\begin{equation}
\label{5.3}
\delta_{s}(|z|)=\int U(y)U(y-z)\mathrm{d}y.
\end{equation}
Following a standard potential analysis, we have
\begin{equation}
\label{5.4}
\delta_{s}(|z|)=O(1)\frac{1}{|z|^{2s+1}+1}.
\end{equation}

For the case $s=\frac12$, we can expand the term $\theta_{ij}(x)$ as follows
\begin{align*}
\theta_{ij}(x)=&-\frac{1}{\pi}\int\log(\varepsilon|x-y|)U(y-q_i)U(y-q_j)\\
&+\mathfrak{a}_2\int U(y-q_i)U(y-q_j)+O(\varepsilon|\log\varepsilon|^{\frac12}).
\end{align*}
Using Proposition \ref{pr4.1} we can obtain that there is a $\delta>0$ and close to $0$ such that
\begin{align*}
\int(\mathfrak{a}_2-\log|x-y|)U(y-q_i)U(y-q_j)=O(\frac{1}{(\log\frac{1}{\varepsilon})^{1-\delta}}),
\end{align*}
uniformly for $x\in I_{\frac12}$, a similar estimate holds for the derivative of the above expression with respect to $x$. Let us set
\begin{equation}
\label{5.5}
\delta_{\frac12}(|z|)=\int U(y)U(y-z)\mathrm{d}y.
\end{equation}
Following a standard potential analysis, we have
\begin{equation}
\label{5.6}
\delta_{\frac12}(|z|)=O(1)\frac{1}{1+|z|^2}.
\end{equation}

Thus, combing the above estimates we obtain:
\begin{lemma}
\label{le5.1}
For the term $V$, we have
\begin{itemize}
  \item [(a)] When $s\in(\frac12,1)$, the following estimate holds for $x\in I_s,$
  \begin{align*}
  V(x)=1+\tau_{\varepsilon}\sum_{i=1}^kH_s(|x-q_i|)+c_s\sum_{i\neq j}\delta_s(|q_i-q_j|)+O(\varepsilon^{2s-1}).
  \end{align*}

  \noindent When $s=\frac12$, the following estimate holds for $x\in I_{\frac12},$
  \begin{align*}
  V(x)=1+\tau_{\varepsilon}\sum_{i=1}^kH_{\frac12}(|x-q_i|)+c_{\frac12}\sum_{i\neq j}\delta_{\frac12}(|q_i-q_j|)+O(\frac{1}{-\log\varepsilon}).
  \end{align*}
  Here $c_s, c_{\frac12}$ are some generic constants independent of $\varepsilon$\\

  A similar estimate holds for the derivatives of $V$ with respect to $x$. The function $H_s(|x|)$ is given in the above discussion and, for $|x|>1$, has the expansion
  \begin{align*}
  H(|x|)=\left\{\begin{array}{ll}
  \mathfrak{a}_1|x|^{2s-1}\int U(|y|)^2+f_s(|x|),~&s\in(0,\frac12),\\
  -\frac{1}{\pi}\log|x|\int U(|y|)^2+f_{\frac12}(|x|),~&s=\frac12.
  \end{array}\right.
  \end{align*}
  The function $\delta_s$ is given in (\ref{5.4}) and (\ref{5.6}) for the cases $s\in(\frac12,1)$ and $s=\frac12$ respectively.
  \item [(b)] When $s\in(\frac12,1)$ for $x\in \mathbb{R}\setminus I_s,$ then we have the following lower estimate
  \begin{equation}
  \label{5.7}
  V(x)\geq C\frac{1}{1+(\varepsilon|x|)^{2s+1}}.
  \end{equation}
  When $s=\frac12,$ for $x\in\mathbb{R}\setminus I_{\frac12},$ then we have the following lower estimate
  \begin{equation}
  \label{5.8}
  V(x)\geq C\frac{1}{1+(\varepsilon|x|)^{2}}.
  \end{equation}
\end{itemize}
\end{lemma}
Estimate (\ref{5.7}) and (\ref{5.8}) can be proven by using the potential analysis.

\vspace{1cm}
\section{further estimates}
For brevity we shall denote $U_i(x)=U(x-q_i)$ and $W=\sum_{i=1}^{k}U_i$. Our purpose in this section is to derive estimates for the quantity
\begin{equation}
\label{6.1}
S(W)\equiv -(-\Delta)^sW-W+\frac{W^2}{V},
\end{equation}
which can be written as
\begin{equation*}
S(W)=\frac{W^2}{V}-\sum_{i=1}^{k}U_i^2.
\end{equation*}
Our first result is the following
\begin{lemma}
\label{le6.1}
Let the number $\mu=1+2s$ in the definition of $\|\cdot\|_*$. For all points $q_i$ satisfying conditions in (\ref{3.4})
and all sufficiently small $\varepsilon$ we have
\begin{equation*}
\|S(W)\|_*\leq
\left\{\begin{array}{lll}
C\varepsilon^{s-\frac{1}{4s}},~&s\in(\frac12,1),\\
C(-\frac{1}{\log\varepsilon})^{1-\delta},~&s=\frac12,
\end{array}\right.
\end{equation*}
where $C$ is some constant independent of $\varepsilon$ and $\delta$ is any small positive number independent of $\varepsilon.$
\end{lemma}
\begin{proof}
Let us assume first $x\in I_s.$ We write
\begin{equation}
\label{6.2}
S(W)=\frac{1-V}{V}\sum_{i=1}^kU_i^2+2V^{-1}\sum_{i\neq j}U_iU_j=J_1+J_2.
\end{equation}
At first, it is easy to see in the region under consideration, we have
\begin{equation*}
V(x)=\left\{\begin{array}{ll}
1+O(\varepsilon^{s-\frac{1}{4s}}),~&s\in(\frac12,1),\\
1+O(\frac{1}{|\log\varepsilon|^{1-\delta}}),~&s=\frac12,
\end{array}\right.
\end{equation*}
where $\delta$ is any small positive number. Hence,
\begin{equation*}
J_1=\left\{\begin{array}{ll}
O(\varepsilon^{s-\frac{1}{4s}})(\sum_{i=1}^kU_i^2),~&s\in(\frac12,1),\\
O(\frac{1}{|\log\varepsilon|^{1-\delta}})(\sum_{i=1}^kU_i^2),~&s=\frac12.
\end{array}\right.
\end{equation*}
On the other hand,
\begin{align*}
V^{-1}U_iU_j\leq 2U_iU_j&\leq C\frac{1}{|q_i-q_j|^{2s+1}}\Big(\frac{1}{(1+|x-q_i|)^{2s+1}}+\frac{1}{(1+|x-q_j|)^{2s+1}}\Big)\\
&=O(1)\frac{1}{|q_i-q_j|^{2s+1}}\rho(x).
\end{align*}
Hence
\begin{equation*}
J_2=\left\{\begin{array}{ll}
O(\varepsilon^{s-\frac{1}{4s}})\rho(x),~&s\in(\frac12,1),\\
O(\frac{1}{|\log\varepsilon|})\rho(x),~&s=\frac12.
\end{array}\right.
\end{equation*}
As a conclusion, for $x\in I_s$, we have
\begin{align}
\label{6.3}
|S(W)|=\left\{\begin{array}{ll}
O(\varepsilon^{s-\frac{1}{4s}})\rho(x),~&s\in(\frac12,1),\\
O(\frac{1}{|\log\varepsilon|^{1-\delta}})\rho(x),~&s=\frac12,
\end{array}\right.
\end{align}
for small $\varepsilon.$

Outside the above region, let us consider the case $s\in(\frac12,1)$ first. Assume now $x\in\mathbb{R}\setminus I_s.$ By using Lemma \ref{le5.1}, we have
\begin{align}
\label{6.4}
|S(W)|\leq&~C\big(\varepsilon^{2s+1}|x|^{2s+1}+1\big)\sum_{i=1}^{k}U_i^2
\leq~C\sum_{i=1}^k\frac{\varepsilon^{2s+1}|x|^{2s+1}+1}{(1+|x-q_i|)^{1+2s}}\rho(x)\nonumber\\
=&~O(\varepsilon^{s-\frac{1}{4s}})\rho(x).
\end{align}
For the case $s=\frac12.$ Assume $x\in\mathbb{R}\setminus I_{\frac12}$, we get
\begin{align}
\label{6.5}
|S(W)|\leq&~C\big(\varepsilon^{2}|x|^{2}+1\big)\sum_{i=1}^{k}U_i^2
\leq ~C\sum_{i=1}^k\frac{\varepsilon^{2}|x|^{2}+1}{(1+|x-q_i|)^{2}}\rho(x)\nonumber\\
=&~O(-\frac{1}{\log\varepsilon})\rho(x)
\end{align}
Combining relations (\ref{6.3}), (\ref{6.4}) and (\ref{6.5}), we prove the lemma.
\end{proof}

Another quantity which will be crucial for the remaining arguments is
\begin{align}
\label{6.6}
\mathcal{I}=\int S(W)Z_{i}.
\end{align}
We shall consider $i=1$ only, since the other cases are similar. Observe that $\frac{\partial U(x-q_1)}{\partial q_1}=-\frac{\partial U(x-q_1)}{\partial x}$ and thus we have
\begin{align*}
-\mathcal{I}=&\int(1-V)V^{-1}\sum_{i=1}^kU_i^2\frac{\partial U}{\partial x}(x-q_1)+\int V^{-1}\sum_{i\neq j}U_iU_j\frac{\partial U}{\partial x}(x-q_1)=I_1+I_2.
\end{align*}
We will estimate $I_1$ and $I_2$ separately. In fact, we will find the following expansions
\begin{align}
\label{6.7}
I_2=-\alpha_s\frac{\partial}{\partial q_1}\sum_{j\neq 1}U(q_j-q_1)(1+o(1)),
\end{align}
and
\begin{align}
\label{6.8}
I_1=\left\{\begin{array}{ll}
-\beta_{s}\varepsilon^{2s-1}\frac{\partial}{\partial q_1}\sum_{j\neq 1}|q_j-q_1|^{2s-1}(1+o(1)),~&s\in(\frac12,1),\\
-\beta_{\frac12}(-\frac{1}{\log\varepsilon})\frac{\partial}{\partial q_1}\sum_{j\neq1}\log|q_j-q_1|(1+o(1)),~&s=\frac12.
\end{array}\right.
\end{align}
Here $\alpha_s,\beta_s,~s\in[\frac12,1)$ are some universal positive constants which are independent of $\varepsilon$.

We will establish (\ref{6.7}) at first. Using Lemma \ref{le5.1} we obtain
\begin{align*}
\int V^{-1}\sum_{i\neq j}U_iU_j\frac{\partial U_1}{\partial x}=\int\sum_{i\neq j}U_iU_j\frac{\partial U_1}{\partial x}(1+o(1)).
\end{align*}
Let us estimate $\int U_iU_j\frac{\partial U_1}{\partial x}$ for $i\neq j$, we observe that if $i,j\neq1,$ then
\begin{align*}
\int U_iU_j\frac{\partial U_1}{\partial x}=O\big((|q_i-q_1||q_j-q_1|)^{-(1+2s)}\big).
\end{align*}
On the other hand, if $i=1,j\neq1$ we get
\begin{align*}
\int U_1U_j\frac{\partial U_1}{\partial x}=-\frac12\frac{\partial}{\partial q_1}\int U^2(x-q_1)U(x-q_j)
\end{align*}
For the right hand side of the above equality, by standard potential analysis, we can get that for a certain universal constant $\mathfrak{c}_s>0$ such that
\begin{align*}
\int U^2(x-q_1)U(x-q_j)=\mathfrak{c}_sU(q_1-q_j)(1+o(1)),
\end{align*}
with a similar estimate for its derivative. We leave the detail in the appendix. Hence,
\begin{align*}
\int V^{-1}\sum_{i\neq j}U_iU_j\frac{\partial U}{\partial x}(x-q_1)=-\frac12\mathfrak{c}_s\frac{\partial }{\partial q_1}\sum_{j\neq 1}U(q_1-q_j)(1+o(1)),
\end{align*}
and estimates (\ref{6.7}) thus follows.

Let us now consider the term $I_1.$ Following a similar procedure, we get
\begin{align*}
I_1=&-\int\sum_{i=1}^{k}U(x-q_i)^2\Big\{\sum_{j=1}^{k}\tau_{\varepsilon} H_s(|x-q_j|)+c_s\sum_{j\neq l}\delta_s(|q_j-q_l|)\Big\}\\
&\times\frac{\partial}{\partial x}U(x-q_1)(1+o(1)),
\end{align*}
where $H_s$ and $\delta_s$ are given in Lemma \ref{le5.1}. Let us first estimate
\begin{align*}
g_{ijl}=\int U(x-q_i)^2\delta_s(|q_j-q_l|)\frac{\partial}{\partial x}U(x-q_1)\mathrm{d}x,
\end{align*}
with $j\neq l$. For $i=1$ this term is zero, while for $i\neq1$ we can estimate, using (\ref{5.4}), (\ref{5.6}) and Lemma \ref{le5.1},
\begin{align*}
|g_{ijl}|\leq~C|q_j-q_l|^{-(2s+1)}(1+|q_i-q_1|)^{-(2s+2)}.
\end{align*}
Then we come to consider the terms
\begin{align*}
I_{ij}=\int U(x-q_i)^2H_s(|x-q_j|)\frac{\partial U}{\partial x}(x-q_1).
\end{align*}
First we observe that the term corresponding to $i=j=1$ vanishes, by symmetry. If $i$ and $j$ are different, and both different from $1$, then the resulting term is of lower order, more precisely
\begin{align*}
I_{ij}=&\int U(x-q_i)^2H_s(x-q_j)\frac{\partial}{\partial x}U(x-q_1)\\
=&-\frac{\partial}{\partial q_1}\int U(x-q_i)^2H_s(x-q_j)U(x-q_1)\\
=&\left\{\begin{array}{ll}
O(1)\frac{\partial}{\partial q_1}\big(|q_i-q_j|^{2s-1}U(q_i-q_1)\big),
~&s\in(\frac12,1),\\
O(1)\frac{\partial}{\partial q_1}\big(\log|q_i-q_j|U(q_i-q_1)\big),
~&s=\frac12.
\end{array}\right.
\end{align*}
On the other hand, if $i=1,$
\begin{align*}
I_{1j}=&-\frac13\frac{\partial}{\partial q_1}\int U(x)^3H_s(|x-(q_j-q_1)|)\mathrm{d}x\\
=&\left\{\begin{array}{ll}
\mathfrak{b}_s\frac{\partial}{\partial q_1} |q_1-q_j|^{2s-1}(1+o(1)),~&s\in(\frac12,1),\\
\mathfrak{b}_{\frac12}\frac{\partial}{\partial q_1}\log|q_1-q_j|(1+o(1)),~&s=\frac12,
\end{array}\right.
\end{align*}
where $\mathfrak{b}_s$ are some generic constants and we used
\begin{equation*}
H_s(x)=
\left\{\begin{array}{ll}
\mathfrak{a}_0\int U^2(y)\mathrm{d}y|x|^{2s-1}(1+o(1)),~&s\in(\frac12,1),\\
-\frac{1}{\pi}\int U^2(y)\mathrm{d}y\log|x|(1+o(1)),~&s=\frac12,
\end{array}\right.
\end{equation*}
provided $x$ is sufficiently large. We put the detail in the appendix. Now, as for $I_{i1}$, we get
\begin{align*}
I_{i1}=\int U(x-(q_i-q_1))^2H_s(|x|)\frac{\partial U}{\partial x}(x)=O(|q_i-q_1|^{-3}).
\end{align*}
Combining the above estimates, we immediately get (\ref{6.8}).

Hence we have found that
\begin{align*}
\int S(W)Z_1=
\left\{\begin{array}{ll}
\sum_{j\neq 1}\frac{\partial}{\partial q_1}\Big(\alpha_{s}U(q_j-q_1)+\beta_s\varepsilon^{2s-1}|q_j-q_1|^{2s-1}\Big)(1+o(1)),~&s\in(\frac12,1),\\
\sum_{j\neq 1}\frac{\partial}{\partial q_1}\Big(\alpha_{\frac12}U(q_j-q_1)+\beta_{\frac12}\frac{\log|q_j-q_1|}{\log\frac{1}{\varepsilon}}\Big)(1+o(1)),
~&s=\frac12.
\end{array}\right.
\end{align*}
Thus, we obtain the following result:
\begin{lemma}
\label{le6.2}
For all points $(q_1,q_2,\cdots,q_k)$ satisfies (\ref{3.4}). If $s\in(\frac12,1),$
\begin{align*}
\int S(W)Z_{j}=\sum_{i\neq j}\frac{\partial F_s(|q_j-q_i|)}{\partial q_j}(1+o(1)),
\end{align*}
where
\begin{align*}
F_s(r)=\alpha_sU(r)+\beta_s\varepsilon^{2s-1}r^{2s-1}.
\end{align*}
If $s=\frac12,$
\begin{align*}
\int S(W)Z_j=\sum_{i\neq j}\frac{\partial F_{\frac12}(|q_j-q_i|)}{\partial q_j}(1+o(1)),
\end{align*}
where
\begin{align*}
F_{\frac12}(r)=\alpha_{\frac12}U(r)+\beta_{\frac12}\frac{\log r}{\log\frac{1}{\varepsilon}}.
\end{align*}
\end{lemma}
\vspace{1cm}

\section{The finite-dimensional reduction}
We will carry cut the finite-dimensional reduction process sketched in the first part of the paper. As in the previous section, we shall assume the points $q_i$ satisfy (\ref{3.4}). Recall from Section 3 that the original problem was cast in the form
\begin{align}
\label{7.1}
(-\Delta)^su+u=\frac{u^2}{T(u^2)}.
\end{align}
Rather than solving this directly we consider instead the problem of finding $A$ such that for certain constants $c_i$ one has
\begin{align}
\label{7.2}
(-\Delta)^sA+A=\frac{A^2}{T(A^2)}+\sum_ic_iZ_i
\end{align}
and $\langle A-W,Z_i\rangle=0$ for all $i$. Rewriting $A=W+\phi$ we get that this problem is equivalent to
\begin{align}
\label{7.3}
&(-\Delta)^s\phi+\phi-2W\phi+2W^2\omega\int W\phi\nonumber\\
=&-(-\Delta)^sW-W+\frac{W^2}{V}+\frac{(W+\phi)^2}{T((W+\phi)^2)}-\frac{W^2}{V}-2W\phi+2W^2\omega\int W\phi+\sum_ic_iZ_i\nonumber\\
=&~S(W)+N(\phi)+\sum c_iZ_i
\end{align}
and
\begin{align}
\label{7.4}
\langle \phi, Z_i\rangle=0~\mathrm{for~all}~i.
\end{align}
Using the operator $\mathcal{T}$ introduced in Theorem \ref{th4.1}, we see that the problem is equivalent to finding a $\phi\in\mathcal{H}$ so that
\begin{equation*}
\phi=\mathcal{T}(S(W)+N(\phi))\equiv Q(\phi).
\end{equation*}
We will show that this fixed point problem has a unique solution in a region of the form
\begin{align}
\label{7.5}
\mathfrak{D}_s=\Big\{\phi\in\mathcal{H}\mid
\left\{\begin{array}{ll}
\|\phi\|_*\leq C\varepsilon^{s-\frac{1}{4s}},~&s\in(\frac12,1),\\
\|\phi\|_*\leq C\frac{1}{|\log\varepsilon|^{1-\delta}},&s=\frac12,
\end{array}\right.
\Big\}
\end{align}
for any small positive constant $\delta$, provided $\varepsilon$ is sufficiently small. Here
\begin{align*}
\mathcal{H}=\{\phi\in L^{\infty}\mid\langle\phi, Z_j\rangle=0,~j=1,2,\cdots,k\}.
\end{align*}

We recall that from Lemma \ref{le6.1},
\begin{align*}
\left\{\begin{array}{ll}
\|S(W)\|_*\leq C\varepsilon^{s-\frac{1}{4s}},~&s\in(\frac12,1),\\
\|S(W)\|_*\leq C\frac{1}{|\log\varepsilon|^{1-\delta}},&s=\frac12.
\end{array}\right.
\end{align*}
On the other hand, $N(\phi)$ admits the estimate provided by the following lemma.
\begin{lemma}
\label{le7.1}
Assume that $\phi\in\mathcal{D}_s$. Then
\begin{align*}
\|N(\phi)\|_*\leq C\big(\|\phi\|_*+\sigma(\varepsilon)\big)\|\phi\|_*
\end{align*}
provided $\varepsilon$ is taken sufficiently small. Here
\begin{align*}
\sigma(\varepsilon)=
\left\{\begin{array}{ll}
\varepsilon^{s-\frac{1}{4s}},~s\in(\frac12,1),\\
\frac{1}{|\log\varepsilon|^{1-\delta}},~s=\frac12,
\end{array}\right.
\end{align*}
as $\varepsilon\rightarrow0$.
\end{lemma}

\begin{proof}
Let us assume first $x\in\mathbb{R}\setminus I_s$ and $s\in[\frac12,1)$. We observe that using a standard potential analysis one can show that in this range of $x$ we have $W(x)\leq C\rho(x)$ and
\begin{align*}
T\big((W+\phi)^2\big)\geq C\frac{1}{1+(\varepsilon|x|)^{2s+1}}~\mathrm{and}~T(W\phi),T(\phi^2)\leq C\frac{1}{1+(\varepsilon|x|)^{2s+1}}\|\phi\|_*.
\end{align*}
Then
\begin{align*}
|N(\phi)|\leq&~
\Big[\frac{2WV\phi+V\phi^2-2W^2T(W\phi)-W^2T(\phi^2)}{VT((W+\phi)^2)}-2W\phi+2\omega W^2\int W\phi\Big]\nonumber\\
\leq&~C\Big[\frac{\rho(x)^2}{1+(\varepsilon|x|)^{2s+1}}
+\frac{\rho(x)^2}{1+(\varepsilon|x|)^{2s+1}}\|\phi\|_*
\Big]\|\phi\|_*+C\rho(x)^2\|\phi\|_*.
\end{align*}
Therefore, we get
\begin{equation}
\label{7.6}
|\rho(x)^{-1}N(\phi)|\leq~C\Big[\|\phi\|_*+\sigma(\varepsilon)\Big]\|\phi\|_*.
\end{equation}

Let us consider now the case $x\in I_s$ for $s\in[\frac12,1)$. We decompose $N(\phi)$ in the form
\begin{equation*}
N(\phi)=N_1(\phi)+N_2(\phi),
\end{equation*}
where
\begin{align*}
N_1(\phi)=(W+\phi)^2\Big[\frac{1}{T((W+\phi)^2)}-\frac{1}{V}+\frac{2T(W\phi)}{V^2}\Big]
-(2W+\phi)\phi\frac{2T(W\phi)}{V^2}
\end{align*}
and
\begin{align*}
N_2(\phi)=-2\phi W(1-\frac{1}{V})+2W^2\big[\omega\int W\phi-\frac{T(W\phi)}{V^2}\big]+\frac{\phi^2}{V}.
\end{align*}
We have that $T((W+\phi)^2)=V+2T(W\phi)+T(\phi^2).$ On the other hand,
\begin{align*}
V(x)=\left\{\begin{array}{ll}
1+O(\varepsilon^{s-\frac{1}{4s}}),~&s\in(\frac12,1),\\
1+O(\frac{1}{|\log\varepsilon|^{1-\delta}}),~&s=\frac12
\end{array}\right.
\end{align*}
Also,
\begin{align*}
T(W\phi)=\left\{\begin{array}{ll}
\omega\int W\phi+O(\varepsilon^{s-\frac{1}{4s}})\|\phi\|_*,~&s\in(\frac12,1),\\
\omega\int W\phi+O(\frac{1}{|\log\varepsilon|^{1-\delta}})\|\phi\|_*,~&s=\frac12,
\end{array}\right.
\end{align*}
and in particular $|T(W\phi)|=O(\|\phi\|_*)$. Likewise, $T(\phi^2)=O(\|\phi\|_*^2)$. Combining these facts we obtain
\begin{align*}
|N_1(\phi)|\leq&~ C(W+\phi)^2T(\phi^2)+C[2W\phi+\phi^2]T(W\phi)\\
\leq&~C\rho(x)\|\phi\|_*^2.
\end{align*}
A similar analysis yields,
\begin{align*}
|N_2(\phi)|\leq&~
\left\{\begin{array}{ll}
C \varepsilon^{s-\frac{1}{4s}}(|\phi|W+W^2\|\phi\|_*)+C|\phi|^2,~&s\in(\frac12,1),\\
\frac{C}{|\log\varepsilon|^{1-\delta}}(|\phi|W+W^2\|\phi\|_*)+C|\phi|^2,~&s=\frac12.
\end{array}\right.
\end{align*}
Hence,
\begin{align*}
\|N(\phi)\|_*\leq C(\|\phi\|_*^2+\sigma(\varepsilon)\|\phi\|_*)
\end{align*}
in this region. Combining this estimate with (\ref{7.6}), yields the result of the lemma.
\end{proof}

Using the definition of the corresponding norms, splitting different ranges of $x$ as in the above proof, it is readily checked that the following holds: If
\begin{equation}
\label{7.7}
\|\phi_i\|_*\leq
\left\{\begin{array}{lll}
C\varepsilon^{s-\frac{1}{4s}},~&s\in(\frac12,1),\\
C(-\frac{1}{\log\varepsilon})^{1-\delta},~&s=\frac12,
\end{array}\right.
\quad i=1,2,
\end{equation}
then, given any small $\mu>0,$ we can find $\varepsilon$ sufficiently small such that
\begin{equation*}
\|N(\phi_1)-N(\phi_2)\|_{*}\leq\mu\|\phi_1-\phi_2\|_*.
\end{equation*}
This implies that the operator $Q$ is a contraction mapping in the set $\mathfrak{D}_s$ defined in (\ref{7.5}). On the other hand, we also get from the Lemma \ref{le7.1} that $Q$ maps $\mathfrak{D}_s$ into itself. By using Banach fixed point theorem, we get the existence of a unique fixed point of $Q$ in this domain, which depends continuously in the $*-$ norm on the points of $q_i.$ We summarize this result in the following proposition:

\begin{proposition}
\label{pr7.1}
For all sufficiently small $\varepsilon$ and all points $q_i$ satisfying (\ref{3.4}), we have the existence of a unique solution to (\ref{7.3}), $\phi_s=\phi_s(q_1,\cdots,q_{k})$ and $\mathbf{c}_s=\mathbf{c}_s(q_1,\cdots,q_{k})$ which satisfies (\ref{7.8}). Besides, $(\phi_s,\mathbf{c}_s)$ depend continuously on the $q_i$'s.

In addition the following formula holds for the components of $c_{s,j}$ of $\mathbf{c}_s$:
\begin{align}
\label{7.8}
c_{s,j}=b_{s,j}+e_{s,j},~j=1,2,\cdots,2m,
\end{align}
with
\begin{align*}
b_{s,j}=\sum_{i\neq j}\frac{\partial F_s(|q_j-q_i|)}{\partial q_j},
\end{align*}
the error terms $e_{s,j}$ satisfy
\begin{align*}
e_{s,j}=
\left\{\begin{array}{ll}
o(1)\varepsilon^{s+\frac12-\frac{1}{2s}},~&s\in(\frac12,1),\\
o(1)\big(\log\frac{1}{\varepsilon}\big)^{-\frac32},~&s=\frac12.
\end{array}\right.
\end{align*}
\end{proposition}

\begin{proof}
We only need to prove the formula for $c_{s,j}$'s. Let us observe that the $c_{s,j}$ satisfy the relations
\begin{align*}
\sum c_{s,j}\langle Z_i,Z_j\rangle=-\langle S(W),Z_j\rangle-\langle N(\phi),Z_j\rangle+\langle\phi,L^*(Z_j)\rangle,
\end{align*}
which define an "almost diagonal" system,  from which the $c_{s,j}$'s can be solved uniquely. The main term in the above expansion is given by $\langle S(W),Z_j\rangle$. To obtain the estimates for these numbers, which will equal the $c_{s,j}$'s at the leading order, we observe that
\begin{align*}
|\langle\phi,L^*(Z_j)\rangle|\leq
\left\{\begin{array}{ll}
C\varepsilon^{\frac{(2s-1)(s+1)}{2s}}\|\phi\|_*,~&s\in(\frac12,1),\\
C\frac{1}{|\log\varepsilon|^{\frac32-\delta}}\|\phi\|_*,~&s=\frac12,
\end{array}\right.
\end{align*}
and
\begin{align*}
|\langle N(\phi),Z_j\rangle|\leq
\left\{\begin{array}{ll}
C\varepsilon^{2s-\frac{1}{2s}},~&s\in(\frac12,1),\\
C\frac{1}{|\log\varepsilon|^{2-\delta}},~&s=\frac12.
\end{array}\right.
\end{align*}
Formula (\ref{7.8}) is now an immediate corollary of Lemma \ref{le6.1}, Lemma \ref{le7.1}, and the expressions found for the $c_{s,j}$'s.
\end{proof}

In the following section we will find the points $q_j$ such that all $c_{s,j}$'s vanish, and satisfying the conditions in (\ref{3.4}).

\vspace{1cm}
\section{Solving the reduced problem}
In this section, we shall look for the point $(q_1,q_2,\cdots,q_k)$ such that $c_{s,j}=0$ and thereby prove Theorem \ref{th1.1}. We first establish the presence of the zero of ${\bf{b}}_s=(b_{s,1},b_{s,2},\cdots,b_{s,k})$ and then use the degree theory to get the existence of the points $(q_1,q_2,\cdots,q_k)$ such that ${\bf{c}}_s=0.$

We recall that
\begin{equation}
\label{8.1}
b_{s,j}=\sum_{i\neq j}\frac{\partial F_s(|q_i-q_j|)}{\partial q_j},
\end{equation}
where
\begin{align*}
F_s(r)=\left\{\begin{array}{ll}
\alpha_sU(r)+\beta_s\varepsilon^{2s-1}r^{2s-1},~&s\in(\frac12,1),\\
\alpha_{\frac12}U(r)+\beta_{\frac12}\frac{\log r}{\log\frac{1}{\varepsilon}},~&s=\frac12.
\end{array}\right.
\end{align*}
It is not difficult to see that finding the zero point of ${\bf{b}}_s$ is equivalent to finding the critical point of the following function,
\begin{align}
\label{8.2}
\widehat{\Xi}_s(q_1,q_2,\cdots,q_k)=
\sum_{i\neq j}\Big(\alpha_sU(q_i-q_j)+\beta_s\varepsilon^{2s-1}|q_i-q_j|^{2s-1}\Big)
\end{align}
for $s\in(\frac12,1)$ and
\begin{align}
\label{8.3}
\widehat{\Xi}_{\frac12}(q_1,q_2,\cdots,q_k)=\sum_{i\neq j}\Big(\alpha_{\frac12}U(q_i-q_j)
+\beta_{\frac12}\frac{1}{\log\frac{1}{\varepsilon}}\log|q_i-q_j|\Big)
\end{align}
for $s=\frac12.$

Since $q_i$ and $q_{k+1-i}$ are symmetry with respect to the origin for $i=1,2,\cdots,k.$ So, finding the critical point of $\widehat{\Xi}_s,~s\in[\frac12,1)$ in (\ref{3.4}) is reduced to finding the critical point of $\Xi_s$ introduced in (\ref{1.9}) for $s\in(\frac12,1)$ in (\ref{1.11}) and (\ref{1.10}) for $s=\frac12$ in (\ref{1.12}) respectively.


For the functions $\Xi_s,~s\in[\frac12,1)$, we have the following property
\begin{lemma}
\label{le8.1}
The functions $\Xi_s,~s\in[\frac12,1)$ admit an interior minimal point in the set (\ref{1.11}) and (\ref{1.12}) for $s\in(\frac12,1)$ and $s=\frac12$ respectively provided $\eta$ is sufficiently small.
\end{lemma}
\begin{proof}
By Proposition \ref{pr4.1}, we have as $|x|\rightarrow\infty,$
$$U(x)\rightarrow\frac{\mathfrak{b}_s}{|x|^{1+2s}}(1+o(1)).$$

For $\varepsilon$ sufficiently small, we have that $\Xi_s$ admit the following asymptotic expansion
\begin{align}
\label{8.4}
\Xi_s(q_1,q_2,\cdots,q_m)=~&
\sum_{i=1}^m\Big(\frac{\mathfrak{b}_s\alpha_s}{|2q_i|^{1+2s}}(1+o(1))+\beta_s\varepsilon^{2s-1}|2q_i|^{2s-1}\Big)\nonumber\\
&+\sum_{i\neq j}\Big(\frac{\mathfrak{b}_s\alpha_s}{|q_i-q_j|^{1+2s}}(1+o(1))+\beta_s\varepsilon^{2s-1}|q_i-q_j|^{2s-1}\Big)\nonumber\\
&+\sum_{i\neq j}\Big(\frac{\mathfrak{b}_s\alpha_s}{|q_i+q_j|^{1+2s}}(1+o(1))+\beta_s\varepsilon^{2s-1}|q_i+q_j|^{2s-1}\Big)
\end{align}
for $s\in(\frac12,1)$ and
\begin{align}
\label{8.5}
\Xi_{\frac12}(q_1,q_2,\cdots,q_m)=~&
\sum_{i=1}^m\Big(\frac{\mathfrak{b}_{\frac12}\alpha_{\frac12}}{|2q_i|^2}(1+o(1))
+\beta_{\frac12}\frac{1}{\log\frac{1}{\varepsilon}}\log|2q_i|\Big)\nonumber\\
&+\sum_{i\neq j}\Big(\frac{\mathfrak{b}_{\frac12}\alpha_{\frac12}}{|q_i-q_j|^2}(1+o(1))
+\beta_{\frac12}\frac{1}{\log\frac{1}{\varepsilon}}\log|q_i-q_j|\Big)\nonumber\\
&+\sum_{i\neq j}\Big(\frac{\mathfrak{b}_{\frac12}\alpha_{\frac12}}{|q_i+q_j|^2}(1+o(1))
+\beta_{\frac12}\frac{1}{\log\frac{1}{\varepsilon}}\log|q_i+q_j|\Big)
\end{align}
for $s=\frac12.$

For convenience, we make the following substitution,
\begin{align}
\label{8.6}
|q_i-q_j|=\left\{\begin{array}{ll}
\varepsilon^{\frac{1-2s}{4s}}|d_i-d_j|,~&s\in(\frac12,1),\\
(\log\frac{1}{\varepsilon})^{\frac12}|d_i-d_j|,~&s=\frac12.
\end{array}\right.
\end{align}

Then, (\ref{8.4})-(\ref{8.5}) turns to be
\begin{align}
\label{8.7}
\Xi_s(q_1,q_2,\cdots,q_m)=~&\mathfrak{b}_s\alpha_s\varepsilon^{s-\frac{1}{4s}}\Big[
\sum_{i=1}^m\Big(\frac{1}{|2d_i|^{1+2s}}(1+o(1))+\gamma_s|2d_i|^{2s-1}\Big)\nonumber\\
&+\sum_{i\neq j}\Big(\frac{1}{|d_i-d_j|^{1+2s}}(1+o(1))+\gamma_s|d_i-d_j|^{2s-1}\Big)\nonumber\\
&+\sum_{i\neq j}\Big(\frac{1}{|d_i+d_j|^{1+2s}}(1+o(1))+\gamma_s|d_i+d_j|^{2s-1}\Big)\Big]
\end{align}
for $s\in(\frac12,1)$ and
\begin{align}
\label{8.8}
\Xi_{\frac12}(q_1,q_2,\cdots,q_m)=~&
\mathfrak{b}_{\frac12}\alpha_{\frac12}\frac{1}{\log\frac{1}{\varepsilon}}\Big[
\sum_{i=1}^m\Big(\frac{1}{|2d_i|^2}(1+o(1))
+\gamma_{\frac12}\log|2d_i|\Big)\nonumber\\
&+\sum_{i\neq j}\Big(\frac{1}{|d_i-d_j|^2}(1+o(1))
+\gamma_{\frac12}\log|d_i-d_j|\Big)\nonumber\\
&+\sum_{i\neq j}\Big(\frac{1}{|d_i+d_j|^2}(1+o(1))
+\gamma_{\frac12}\log|d_i+d_j|\Big)\nonumber\\
&+(m^2-\frac12m)\gamma_{\frac12}\log(\log\frac{1}{\varepsilon})
\Big]
\end{align}
for $s=\frac12.$ Here $\gamma_s=\frac{\alpha_s}{\beta_S}.$

Since the process of finding the interior global minimal point of $\Xi_s$ for $s\in(\frac12,1)$ in (\ref{1.11}) and $\Xi_s$ for $s=\frac12$ in (\ref{1.12}) are the same, in the following we shall only give the detail of the case $s=\frac12.$

Before studying $\Xi_\frac12$, we first consider the following function
$$\mathfrak{g}(x)=x^{-2}+\gamma_{\frac12}\log x~\mathrm{for}~x>0.$$
By analyzing the derivative of function $\mathfrak{g},$ we know that
$$\mathfrak{g}(\sqrt{\frac{2}{\gamma_{\frac12}}})=\min_{x>0}\mathfrak{g}(x)
=\big(\frac12+\frac12\log2\big)\gamma_{\frac12}-\frac12\gamma_{\frac12}\log\gamma_{\frac12}.$$

Let us come back to the function $\Xi_{\frac12}$. For $(q_1,q_2,\cdots,q_m)\in\partial Q_{s,\eta},$ we have either there is some $i$ such that $d_i=\eta$ or $d_i=\frac{1}{\eta}$, or there are $i,j$ such that $|d_i-d_j|=\eta.$ If the former case happens, i.e., there is some $i$ such that $d_i=\eta$ or $d_i=\frac{1}{\eta}$. For convenience, we write
$$\tilde{\Xi}_{\frac12}(d_1,d_2,\cdots,d_m)=
\Big(\mathfrak{b}_{\frac12}\alpha_{\frac12}\frac{1}{\log\frac{1}{\varepsilon}}\Big)^{-1}\Xi_{\frac12}(q_1,q_2,\cdots,q_m)
-(m^2-\frac12m)\gamma_{\frac12}\log\log\frac{1}{\varepsilon}.$$
Then, we can get
\begin{align}
\label{8.9}
\tilde{\Xi}_{\frac12}(d_1,d_2,\cdots,d_m)\geq&\min\{\frac14\eta^{-2}\big(1+o(1)\big)+\gamma_{\frac12}\log 2\eta,\frac14\eta^2\big(1+o(1)\big)-\gamma_{\frac12}\log\frac{\eta}{2}\}\nonumber\\
&~+\big(2m^2-m-1\big)\big[\big(\frac12+\frac12\log2\big)\gamma_{\frac12}-\frac12\gamma_{\frac12}\log\gamma_{\frac12}
\big]\big(1+o(1)\big)\nonumber\\
\geq&~\tilde{\Xi}_{\frac12}(1,2,\cdots,m)
\end{align}
provided $\eta$ is small enough. If there are some $i,j$ such that $|d_i-d_j|=\eta.$ Then,
\begin{align}
\label{8.10}
\tilde{\Xi}_{\frac12}(d_1,d_2,\cdots,d_m)\geq~&\big(2m^2-m-1\big)\big[\big(\frac12+\frac12\log2\big)\gamma_{\frac12}
-\frac12\gamma_{\frac12}\log\gamma_{\frac12}\big]\big(1+o(1)\big)\nonumber\\
&+\eta^{-2}\big(1+o(1)\big)+\gamma_{\frac12}\log\eta\nonumber\\
\geq~&\tilde{\Xi}_{\frac12}(1,2,\cdots,m)
\end{align}
provided $\eta$ is small enough. By (\ref{8.9}) and (\ref{8.10}), we get
\begin{align}
\label{8.11}
\min_{{\bf{q}}\in\partial Q_{\frac12,\eta}}\Xi_{\frac12}(q_1,q_2,\cdots,q_m)>
\min_{{\bf{q}}\in Q_{\frac12,\eta}}\Xi_{\frac12}(q_1,q_2,\cdots,q_m).
\end{align}
As a conclusion, $\Xi_{\frac12}(q_1,q_2,\cdots,q_m)$ admits a global interior minimal point in $Q_{\frac12,\eta}.$
\end{proof}

Before we give the proof of Theorem \ref{th1.1},  we recall the following definition (see Definition 2.4 in \cite{mp} or in \cite{l}).\\

\noindent{\bf{Definition 8.1}}. Let $f:D\rightarrow\mathbb{R}$ be a $C^1-$function, where $D\subset\mathbb{R}^m$ is an open set. We say that $x_0$ is stable critical point of $f$ if $\nabla f(x_0)=0$ and there exists a neighborhood $U$ of $x_0$ such that
$$\nabla f(x)\neq0,~\forall x\in\partial U,$$
$$\nabla f(x)=0,~x\in U\quad\Longleftrightarrow\quad f(x)=f(x_0),$$
and
$$\mathrm{deg}(\nabla f,U,0)\neq0,$$
where $\mathrm{deg}$ denotes the Brouwer degree.\\

\noindent{\bf{Remark}}: It is easy to see that, if $x_0$ is a global minimum point or a global maximum point of the function $f$, then $x_0$ is a stable critical point of $f$.\\

\noindent{\em Proof of Theorem \ref{th1.1}}: By Lemma \ref{le8.1}, we get the existence of $(q_1,q_2,\cdots,q_k)$ such that
$$\mathbf{b}_s(q_1,q_2,\cdots,q_k)=0.$$
Furthermore, such a point is the global minimal point of the function $\Xi_s$ in (\ref{1.11}) and (\ref{1.12}) for $s\in(\frac12,1)$ and $s=\frac12$ respectively. We denote such point by ${\bf{q}}_s=(q_{s,1},q_{s,2},\cdots,q_{s,m})$. Now, we shall look for ${\bf{q}}=(q_1,q_2,\cdots,q_m)$ in the neighborhood of ${\bf{q}}_s$ to make ${\bf{c}}_s=0$. As we mentioned in the previous remark, such a global minimal point is a stable critical point of $\Xi_s.$ Using the property of the stable critical point and relation between the function $\widehat{\Xi}_s$ and $\Xi_s$, we can find an open neighborhood $O_{{\bf{q}}_s}$ of ${\bf{q}}_s$ in $Q_{s,\eta}$ such that the following holds
\begin{align*}
\mathrm{deg}\big(\nabla_{{\bf{q}}_s}\widehat{\Xi}_s,O_{{\bf{q}}_s},0\big)\neq0~\mathrm{and}~ \nabla_{{\bf{q}}_s}\widehat{\Xi}_s\neq0~\mathrm{on}~\partial O_{{\bf{q}}_s},
\end{align*}
which implies
\begin{align*}
\mathrm{deg}\big({\bf{b}}_s,O_{{\bf{q}}_s},0\big)\neq0~\mathrm{and}~{\bf{b}}_s\neq0~\mathrm{on}~\partial O_{{\bf{q}}_s}.
\end{align*}
According to the definition of $\Xi_s$, we can get
\begin{align*}
|{\bf{b}}_s|\geq
\left\{\begin{array}{ll}
\mathfrak{d}_s\varepsilon^{s+\frac12-\frac{1}{2s}},~&s\in(\frac12,1),\\
\mathfrak{d}_{\frac12}\frac{1}{(\log\frac{1}{\varepsilon})^{\frac32}},~&s=\frac12,
\end{array}\right.~\mathrm{on}~\partial O_{{\bf{q}}_s},
\end{align*}
where $\mathfrak{d}_s,~s\in[\frac12,1)$ are strictly positive constants.

Next, Let us introduce the following homotopy,
\begin{align*}
H(t,s,{\bf{q}})={\bf{b}}_s+t({\bf{c}}_s-{\bf{b}}_s).
\end{align*}
We find that
$$H(1,s,{\bf{q}})={\bf{c}}_s~\mathrm{and}~ H(0,s,{\bf{q}})={\bf{b}}_s.$$
It is known in Proposition \ref{pr7.1}
\begin{align*}
|{\bf{c}}_s-{\bf{b}}_s|=
\left\{\begin{array}{ll}
o(1)\varepsilon^{s+\frac12-\frac{1}{2s}},~&s\in(\frac12,1),\\
o(1)\big(\log\frac{1}{\varepsilon}\big)^{-\frac32},~&s=\frac12,
\end{array}\right.~\mathrm{in}~\overline{O_{{\bf{q}}_s}}.
\end{align*}
As a result, we get $H(t,s,{\bf{q}})\neq0$ on $\partial O_{{\bf{q}}_s}$. Therefore,
$$\mathrm{deg}\big(H(1,s,{\bf{q}}),O_{{\bf{q}}_s},0\big)=\mathrm{deg}\big(H(0,s,{\bf{q}}),O_{{\bf{q}}_s},0\big),$$
which implies
$$\mathrm{deg}\big({\bf{c}}_s,O_{{\bf{q}}_s},0\big)=\mathrm{deg}\big({\bf{b}}_s,O_{{\bf{q}}_s},0\big).$$
We already know that the right hand side of the above equality is non-zero, therefore, $\mathrm{deg}\big({\bf{c}}_s,O_{{\bf{q}}_s},0\big)\neq0.$ As a result, we can find ${\bf{q}}$ in $O_{{\bf{q}}_s}$ such that ${\bf{c}}_s=0$. Hence, we finish the proof of Theorem \ref{th1.1}. $\quad\quad\quad\quad\quad\quad\quad\quad\quad\quad\quad\quad\quad\quad\quad\quad\quad\quad\quad\quad\square$
\vspace{0.3cm}

The proof of even number bumps case is thus concluded. Let as assume that $k=2m+1$ and briefly sketch the way to proceed in this situation. In this case we introduce the set
\begin{align*}
(q_1,q_2,\cdots,q_k)\in \Lambda_s^o=\Big\{&(q_1,q_2,\cdots,q_k)\mid q_k=0,~ q_i=-q_{k-i},~q_1>q_2>\cdots>q_{k-1},\nonumber\\
&(q_1,q_2,\cdots,q_m)\in Q_{s,\eta}\Big\}.
\end{align*}
Now we consider the first approximation
\begin{align*}
W(x)=\sum_{i=1}^kU(x-q_i),~(q_1,q_2,\cdots,q_k)\in\Lambda_s^o.
\end{align*}
In this case, because of the even symmetry, the "bad directions" corresponding to small eigenvalues are only those $Z_i$ with $1\leq i\leq k-1.$ With this observation made, the rest of the scheme of proof goes almost the same way.

\vspace{1cm}
\section{Appendix}
In this section, we list the estimates used in previous sections and give a proof in the following lemma.
\begin{lemma}
\label{lea1}
\begin{align}
\label{a1}
\int U^2(|y|)U(|x-y|)\mathrm{d}y=\mathfrak{c}_sU(|x|)(1+o(1))~\mathrm{as}~|x|\rightarrow\infty,
\end{align}
\begin{align}
\label{a2}
\int U^2(|y|)|x-y|^{2s-1}\mathrm{d}y=\mathfrak{c}_s|x|^{2s-1}(1+o(1))~\mathrm{as}~|x|\rightarrow\infty,
\end{align}
and
\begin{align}
\label{a3}
\int U^2(|y|)\log|x-y|\mathrm{d}y=\mathfrak{c}_s\log|x|(1+o(1))~\mathrm{as}~|x|\rightarrow\infty.
\end{align}
Here $\mathfrak{c}_s$ depends on the integral of $\int U^2$.
\end{lemma}

\begin{proof}
Since the proof of (\ref{a1})-(\ref{a3}) are the same, we only give the proof of the first one.
We divide the whole space into two parts,
\begin{align*}
\mathbb{R}=I_1\cup I_2:=\{y:|y|\leq|x|^{\frac23}\}\cup\{y:|y|>|x|^{\frac23}\}.
\end{align*}
Then, for any $y\in I_1,$ we have
\begin{align*}
\frac{1}{|x-y|^{2s+1}}=&\frac{1}{|x|^{2s+1}}\frac{1}{|1-\frac{y}{x}|^{2s+1}}
=\frac{1}{|x|^{2s+1}}\big(1+O(\frac{y}{x})\big),
\end{align*}
and
\begin{align*}
\int_{I_1}U(y)^2\mathrm{d}y=&\int_{\mathbb{R}^1}U(y)^2\mathrm{d}y+O(\int_{I_2}\frac{1}{|y|^{4s+2}}\mathrm{d}y)\nonumber\\
=&\mathfrak{c}_0+O(|x|^{-\frac{8s+2}{3}}),
\end{align*}
where we used $U(y)\rightarrow\frac{\mathfrak{b}_s}{|y|^{1+2s}}$ as $|y|\rightarrow\infty$ and $\mathfrak{c}_0=\int_{\mathbb{R}^1}U(y)^2\mathrm{d}y.$ Thus,
\begin{align}
\label{a4}
\int_{I_1}U^2(|y|)U(|x-y|)\mathrm{d}y=&U(x)\int_{I_1}U(y)^2+\int_{I_1}\big(U(x-y)-U(x)\big)U(y)^2\nonumber\\
=&\mathfrak{c}_0U(|x|)+o(1)(|x|^{-2s-1}).
\end{align}
For $y\in I_2,$ we have
\begin{equation*}
U(y)^2U(x-y)\leq C\frac{1}{|y|^{4s+2}}U(x-y).
\end{equation*}
Hence,
\begin{equation}
\label{a5}
\int_{I_2}U(y)^2U(x-y)=O(|x|^{-\frac{8s+4}{3}})=O(|x|^{-2s-1-\frac{2s+1}{3}}).
\end{equation}
Combining (\ref{a4}) and (\ref{a5}), we get (\ref{a1}).
\end{proof}

%

\end{document}